\documentclass[11pt]{amsart}
\usepackage[T1]{fontenc}
\usepackage[osf,sc]{mathpazo}
\usepackage{amsmath,amsthm,amsfonts,latexsym,amscd,amssymb,amsopn,enumerate,hyperref, mathrsfs}
\usepackage[utf8]{inputenc}
\usepackage{amsmath}
\usepackage{tikz-cd}
\usepackage{bm}
\usepackage{graphicx}
\usepackage{geometry}
\hypersetup{
    colorlinks=true,
    linktoc=all,
    linkcolor=blue,
    citecolor=red,
    filecolor=black,
    urlcolor=blue
}

\usepackage{geometry}
\usepackage[inline]{enumitem}
\makeatletter
\newcommand{\inlineitem}[1][]{%
    \ifnum\enit@type=\tw@
    {\descriptionlabel{#1}}
    \hspace{\labelsep}%
    \else
    \ifnum\enit@type=\z@
    \refstepcounter{\@listctr}\fi
    \quad\@itemlabel\hspace{\labelsep}%
    \fi} \makeatother
\newtheorem{thm}{Theorem}[subsection]
\newtheorem{lem}[thm]{Lemma}

\newtheorem{cor}[thm]{Corollary}

\newtheorem{example}[thm]{Example}

\makeatletter
\def\namedlabel#1#2{\begingroup
    \def\@currentlabel{#2}%
    \label{#1}\endgroup
}
\makeatother

\theoremstyle{definition}
\newtheorem{definition}[thm]{Definition}

\theoremstyle{remark}
\newtheorem{remark}[thm]{Remark}
\numberwithin{equation}{subsection}
\begin{document}

\title[Rational and Quasi-Permutation Representations]{Rational and Quasi-Permutation Representations of Holomorph of Cyclic $p$-Groups}

\author[Soham Swadhin Pradhan]{Soham Swadhin Pradhan}
\address{Soham Swadhin Pradhan, School of Mathematics, Harish-Chandra Research Institute, HBNI, Chhatnag Road, Jhunsi, Allahabad, 211019,  India. \, {\it Email address:} {\tt soham.spradhan@gmail.com}}

\author[B. Sury]{B. Sury{*}}
\address{B. Sury, Stat-Math Unit, Indian Statistical Institute, Bangalore Centre, 8-th Mile Mysore Road , Bangalore, 560059, India. \, {\it Email address:} {\tt surybang@gmail.com}}

\thanks{* Corresponding author.\\
E-mail addresses: surybang@gmail.com (B. Sury), soham.spradhan@gmail.com (Soham Swadhin Pradhan).}
\date{\sc \today}
\begin{abstract}
For a finite group $G$, let $p(G)$ denote the {\it minimal degree}
of a faithful permutation representation of $G$. The minimal degree
of a faithful representation of $G$ by quasi-permutation matrices
over the fields $\mathbb{C}$ and $\mathbb{Q}$ are denoted by $c(G)$
and $q(G)$ respectively. In general $c(G)\leq q(G)\leq p(G)$ and
either inequality may be strict. In this paper, we study the
representation theory of the group $G =$ Hol$(C_{p^{n}})$, which is
the {\it holomorph} of a cyclic group of order $p^n$, $p$ a prime.
This group is metacyclic when $p$ is odd and metabelian but not
metacyclic when $p=2$ and $n \geq 3$. We explicitly describe the set
of all {\it isomorphism types} of irreducible representations of $G$
over the field of complex numbers $\mathbb{C}$ as well as the
isomorphism types over the field of rational numbers $\mathbb{Q}$.
We compute the {\it Wedderburn decomposition} of the rational group
algebra of $G$. Using the descriptions of the irreducible
representations of $G$ over $\mathbb{C}$ and over $\mathbb{Q}$, we
show that $c(G) = q(G) = p(G) = p^n$ for any prime $p$. The proofs
are often different for the case of $p$ odd and $p=2$.
\end{abstract}
\maketitle
\section{Introduction} \label{S:intro}
Throughout this paper, $G$ always denotes a finite group, $F$ denotes a field of characteristic $0$ and $p$ denotes a prime number.

The {\it holomorph} of a group $G$ is the semi-direct product of the group $G$ with its automorphism group, with respect to the obvious natural action, and is denoted by Hol$(G)$.

Let $C_{p^n}$ denote the cyclic group of order $p^n$, $n \geq 1$. Let $\phi$ denote the Euler's phi function. For $p$ odd, the automorphism group of $C_{p^n}$, is well known to be cyclic group of order $\phi(p^n)$. Therefore, for $p$ odd, Hol$(C_{p^n})$ is a {\it split metacyclic} group. It is well known that the automorphism group of $C_{2^n}$ is isomorphic to the multiplicative group $(\mathbb{Z}/2^{n}\mathbb{Z})^{*}$ of units of the ring $\mathbb{Z}/2^{n}\mathbb{Z}$. For $n = 2$, the multiplicative group $(\mathbb{Z}/4\mathbb{Z})^{*}$ of units of the ring $\mathbb{Z}/4\mathbb{Z}$ is a cyclic group of order $2$, generated by the class $-1$. For $n \geq 3$, the multiplicative group $(\mathbb{Z}/2^{n}\mathbb{Z})^{*}$ of units of the ring $\mathbb{Z}/2^{n}\mathbb{Z}$ is direct product of $\mathbb{Z}/2\mathbb{Z}$, generated by the class of $-1$ and a copy of $\mathbb{Z}/2^{n-2}\mathbb{Z}$, generated by the class of $5$. Note that for $n =2$, Hol$(C_{2^n})$ is the {\it dihedral group} of order $8$, which is a split metacyclic group. For $n \geq 3$, Hol$(C_{2^n})$ is a {\it split metabelian} group but not metacyclic.

Frobenius studied representations of finite groups over the field of complex numbers. Schur extended that to study representations of finite groups over subfields of complex numbers, especially the field of rational numbers. He understood that representations of finite groups over the field of rational numbers, in general subfields of complex numbers, in much deeper and involved arithmetic aspects. Schur observed that certain irreducible representations over the field of complex numbers could not be realized over the field of rational numbers, but that a finite
multiple of certain irreducible representations could always be realized. The number
associated with this multiplicity is called the {\it Schur index} (see Section \ref{Section $2.1$}).

A fundamental result in non-commutative ring theory is the {\it
Artin-Wedderburn theorem} (see \cite{Artin}, \cite{Wedderburn}). The theorem asserts that the semisimple
group ring $F[G]$ is abstractly isomorphic to direct sum of matrix
rings over finite dimensional division algebras over $F$. This is
referred to as the {\it Wedderburn decomposition} of $F[G]$. The
{\it Wedderburn theory}, which came around $1914$, shows that questions on representations of a group $G$ over $F$ can be thought of as questions on the algebraic
structure of the group algebra $F[G]$.

In this paper, we first determine the set of all isomorphism types
of irreducible representations of Hol$(C_{p^n})$ over the field of
complex numbers. Thereafter, using the classical Schur theory on
group representations (see \cite{Schur}), we find the set of all
isomorphism types of irreducible representations of Hol$(C_{p^n})$
over the field of rational numbers. We prove that the Schur index of
any irreducible complex representation of Hol$(C_{p^n})$ with
respect to the field of rationals is equal to $1$. In addition, we
also compute the Wedderburn decomposition of rational group algebra
of Hol$(C_{p^n})$. Note that for $ p = 2, n = 1$, Hol$(C_{p^n})$ is $C_{2}$, and for $ p = 2, n = 2$, that is $D_{8}$, the dihedral group of order $8$. For these two groups, all isomorphism types of irreducible representations and Wedderburn decomposition are well known to us. In this paper, when $p = 2$, we consider $n \geq 3$.

Let $S_{n}$ denote the {\it permutation group } on $n$ letters. Let $|G|$ denote the order of the group $G$. The {\it Cayley’s theorem} states that any group $G$ can be embedded in $S_{|G|}$. It is interesting to find the least positive integer $n$ such that $G$ is embedded in $S_{n}$. The minimal degree of a {\it faithful permutation} representation of $G$ is the least positive integer $n$ such that $G$ is embedded in $S_{n}$, and is denoted by $p(G)$.

In $1963$, analogous to permutation group, Wong (see \cite{wong}) defined a quasi-permutation group as follows: If $G$ is a finite linear group of degree $n$, that is, a finite group of automorphisms of an $n$-dimensional complex vector space (or, equivalently, a finite group of non‐singular matrices of order $n$ with complex coefficients) such that the trace of every element of $G$ is a non-negative integer, then $G$ is called a {\it  quasi-permutation group}. The reason for this terminology is that, if $G$ is a permutation group of degree n, its elements, considered as acting on the elements of a basis of an $n$-dimensional complex vector space, induce automorphisms of the vector space forming a group isomorphic to $G$ . The trace of the automorphism corresponding to an element x of G is equal to the number of letters left fixed by $x$, and so is a non-negative integer. Therefore a permutation group of degree $n$ has a representation as a quasi-permutation group of degree $n$.

A square matrix over $\mathbb{C}$ is called {\it quasi-permutation} matrix if it has non-negative integral trace. Thus every permutation matrix over $\mathbb{C}$ is a quasi-permutation matrix. For a group $G$, let $q(G)$ denote the minimal degree of a faithful representation of $G$ by {\it quasi-permutation matrices over $\mathbb{Q}$}, and let $c(G)$ be the minimal degree of a faithful representation of $G$ by {\it quasi-permutation matrices over $\mathbb{C}$} (see \cite{Burns}). It  is  easy  to see that for a group $G$ the following inequalities hold:
$$c(G)\leq q(G)\leq p(G).$$

In \cite{Burns}, the case of equality has been investigated for abelian groups. For an abelian group, the invariants $c(G)$ and $q(G)$ coincide, as the Schur index of any irreducible complex representation of an abelian group is equal to $1$. In \cite{Behravesh}, Behravesh and Ghaffarzadeh proved that if $G$ is a finite $p$-group then $q(G)=p(G)$. They also proved that for a finite $p$-group $G$, $p$ odd prime, $q(G)=p(G)=c(G)$. The above quantities have been found for several class of finite groups (see \cite{Behravesh7, Behravesh3, Behravesh4, Behravesh5, Behravesh6, Darafsheh, M. Ghaffarzadeh}).

In this paper, for $G =$ Hol$(C_{p^{n}})$, $p$ a prime, we compute
$c(G), q(G), p(G)$, and prove that $c(G) = q(G) = p(G) = p^{n}$. A
point to be noted is that even though the structure of the holomorph
is different for the case $p=2$ from the case of odd primes $p$, the
end result is the same. The method of proof is accomplished by
explicitly describing the equivalence classes of irreducible
representations over $\mathbb{C}$ and over $\mathbb{Q}$, and by
computing the Wedderburn decomposition of the rational group
algebra. To compute $c(G)$ and $q(G)$, we use the basic method
described by Behravesh in \cite{Behravesh7}. \\
{\it The main theorems are Theorems \ref{Theorem $6.0.1$}, \ref{Theorem $7.0.2$}, \ref{Theorem $8.0.1$}, \ref{Theorem $9.1.1$}, \ref{Theorem $9.2.2$}, \ref{Theorem $9.3.1$}, and \ref{Theorem $10.2.2$}.}

\section{Schur Index in the Theory of Group Representations}\label{Section $2.1$}
\vskip 5mm

Let $G$ be a group, and let $F$ be a field of
characteristic $0$. Fix an algebraic closure $\overline{F}$ of $F$.
Consider an irreducible representation $\widetilde \rho$ of $G$ over
$\overline{F}.$ If $u$ is the exponent of $G$, then the character
$\widetilde \chi$ of $\widetilde\rho$ takes values in the field
$F(\zeta_u),$ where $\zeta_u$ is a primitive $u$-th root of unity.
The {\it character field} $F(\widetilde \chi)$ of $\widetilde{\rho}$
over $F$ is the extension field of $F$ obtained by adjoining all the
values $\widetilde \chi(g)$'s as $g$ varies in $G.$ Let $\Gamma$ be
the Galois group of $F(\widetilde \chi)$ over $F.$ It is well known
and easy to see that for each $\sigma,$ an element of $\Gamma,$ the
values $\sigma(\widetilde \chi(g))$ are also values of a character
of a representation of $G$ over $\overline{F}.$ In this manner,
starting with $\widetilde \rho,$ we have obtained $|\Gamma|$
distinct representations  of $G,$ over $\overline{F};$ these
representations are said to be {\it Galois conjugate}. This gives a
class of mutually inequivalent Galois conjugate representations of
$G$ which have different characters, but all of them have the same
character field. Schur's theory implies that there exists a unique
irreducible $F$-representation $\rho$ of $G$ such that $\widetilde
\rho$ occurs as an irreducible constituent of $\rho\otimes_F
{\overline F}$, with some multiplicity; this multiplicity is called
the {\bf Schur index} of $\widetilde \rho$ with respect to $F$. This
Schur index is denoted by $m_F(\widetilde \rho)$ or $m_F(\widetilde
\chi)$. We summarize the above discussion in the following theorem.
\vskip 3mm

\begin{thm} [\cite{MR0122892}, Theorem 2] \label{Theorem $1$}
{(\bf Schur)} Every irreducible $F$-representation $\rho$ is completely reducible over $\overline{F}$ into a certain number inequivalent irreducible $\overline{F}$-representations $\widetilde{\rho} = \widetilde{\rho_{1}}, \widetilde{\rho_{2}}, \dots, \widetilde{\rho_{\delta}}$ with the same multiplicity $m$ given by
$$m = m_{F}(\widetilde \chi_{i}), (1 \leq i \leq \delta),$$
where $\widetilde{\chi}_{i}$ is the character of $\widetilde{\rho_{i}}$ with $\widetilde \chi_{1} = \widetilde \chi$. The number $\delta$ is given by
$$\delta = [F(\widetilde \chi_{i}): F], (1 \leq i \leq \delta).$$ The $\overline{F}$-representations $\widetilde{\rho} = \widetilde{\rho_{1}}, \widetilde{\rho_{2}}, \dots, \widetilde{\rho_{\delta}}$ are Galois conjugates with respect to $F$. Conversely, each irreducible $\overline{F}$-representation $\widetilde{\rho}$ occurs in the decomposition of an unique irreducible $F$-representation $\rho$ of $G$. Moreover, if $\psi$ is the character corresponding to the irreducible $F$-representation $\rho$, and $\Gamma$ is the Galois group of $F(\widetilde{\chi_{i}})$ over $F$, then
$$\psi = m(\sum_{\sigma \in \Gamma} {\widetilde{\chi}}^\sigma),$$
where $\widetilde{\chi}^\sigma(g) = \sigma(\widetilde{\chi}(g))$, for all $g \in G$.
\end{thm}

\begin{definition}
Let $\rho$ be an irreducible $F$-representations of $G$, and $\chi$ be its corresponding character. The {\it kernel} of $\rho$ or $\chi$ is defined as $\{g \in G ~|~ \chi(g) = \chi(1)\}$, and is denoted by ker$(\rho)$ or ker$(\chi)$.
\end{definition}
\begin{thm}\label{Theorem $2.0.3$}
$G$ has a faithful irreducible $F$-representation if  and only if $G$ has a faithful irreducible $\overline{F}$-representation.
\end{thm}
\begin{proof}
Let $\rho$ be a faithful irreducible $F$-representation of $G$, and $\psi$ be its corresponding character. By Theorem \ref{Theorem $1$}, we have
$$\psi = m(\sum_{\sigma \in \Gamma} { \widetilde{\chi}}^\sigma),$$
where $\widetilde{\chi}$ is an irreducible $\overline{F}$-character of $G$ and $\Gamma$ is the Galois group of $F(\widetilde{\chi_{i}})$ over $F$. Then it is easy to see that ker$(\psi) =$ ker$(\widetilde \chi)$. This completes the proof.
\end{proof}
\vskip 5mm

\section{Algorithms for $p(G)$, $c(G)$ and $q(G)$}
\vskip 5mm

\noindent We begin with some basic results on $p(G)$, $c(G)$ and
$q(G)$ due to Behravesh \cite{Behravesh7} for a finite group $G$.
Recall that $p(G)$ denotes the {\it minimal degree} of a faithful
permutation representation of $G$, and that the minimal degrees of a
faithful representation of $G$ by quasi-permutation matrices over
the fields $\mathbb{C}$ and $\mathbb{Q}$ are denoted by $c(G)$ and
$q(G)$ respectively.

\begin{lem}[\cite{Behravesh7}, Lemma 2.1]
Let $G$ be a finite group. Let $H_{i} \leq G$ for $i = 1, 2, \dots , n$, and $K_{i} = \cap_{x \in G} H^{x}_{i}$, the core of $H_{i}$ in $G$. Suppose that $ \displaystyle \cap^{n}_{i=1} K_{i} = \{1\}$. Then $\displaystyle p(G) \leq \sum^{n}_{i = 1}[G:H_{i}]$.
\end{lem}
\begin{thm}[\cite{Behravesh7}, Theorem 2.2]
Let $G$ be a finite group. Then
$$\displaystyle p(G) = \textrm{min}\Big \{ \sum^{n}_{i = 1}[G : H_{i}]| H_{i} \leq G (1 \leq i \leq n), \cap^{n}_{i = 1}\mathrm{core}_{G}{H_{i}} = 1 \Big\}.$$
\end{thm}
\begin{definition}
{\it Let $G$ be a finite group and $\chi$ be a irreducible complex character of $G$. Let $\Gamma(\chi)$ be the Galois group of $\mathbb{Q}(\chi)$ over $\mathbb{Q}$. We define
    \begin{itemize}
        \item[(1)] $\displaystyle d(\chi) = |\Gamma(\chi)|\chi(1).$
        \item [(2)]
        \begin{equation*}
        \displaystyle m(\chi) =
        \begin{cases}
        0  &\mathrm{if} \, \chi = 1_G\\
        \Big{|}\mathrm{min} \Big{\{}\displaystyle \sum_{\sigma \in \Gamma(\chi)}{\chi^{\sigma}(g) : g \in G}\Big{\}}\Big{|}& \mathrm{otherwise}.\\
        \end{cases}
        \end{equation*}
        \item[(3)] $\displaystyle c(\chi) = \sum_{\sigma \in \Gamma(\chi)}{\chi^{\sigma}} + m(\chi)1_{G}.$
    \end{itemize}}
\end{definition}
\begin{lem}
Let $\chi$ be an irreducible complex character of $G$. Then $m_{\mathbb{Q}}(\chi)\sum_{\sigma \in \Gamma(\chi)}{\chi^{\sigma}}$, where $m_{\mathbb{Q}}(\chi)$ is the Schur index of $\chi$ with respect to $\mathbb{Q}$, is an irreducible rational character of $G$. Moreover, ker$(\chi) =~$ker$(m_{\mathbb{Q}}(\chi)\sum_{\sigma \in \Gamma(\chi)}{\chi^{\sigma}})$, and in particular, $\chi$ is faithful if and only if $m_{\mathbb{Q}}(\chi)\sum_{\sigma \in \Gamma(\chi)}{\chi^{\sigma}}$ is also faithful.
\end{lem}
The proof follows from Theorem \ref{Theorem $1$} and Theorem \ref{Theorem $2.0.3$}.

With notations as in above, let $\{\mathcal{C}_{i} \,|\, 0 \leq i \leq r\}$ be the set of Galois conjugacy classes of irreducible complex characters of $G$, over $\mathbb{Q}$. For each $i \in \{0, 1, 2, \dots, r \}$, let $\psi_{i}$ be a representative of the Galois conjugacy class $\mathcal C_{i}$ with $\psi_{0} = 1_{G}$. Let $m_{i} = m_{\mathbb{Q}}(\psi_{i})$ be the Schur index of $\psi_{i}$ with respect to $\mathbb{Q}$ and $K_{i} = \textrm{ker} \psi_{i}$ be the kernel of $\psi_{i}$. Note that both the Schur index and kernel is independent of the choice of the representatives. Let $\Psi_{i} = \sum_{\chi \in \mathcal{C}_{i}}{\chi}$. Clearly $K_{i} = \textrm{ker} \Psi_{i}$. For $I \subseteq \{0, 1, \dots , r\}$, let $K_{I} = \cap_{i \in I}K_{i}$.
\begin{thm}[\cite{Behravesh7}, Theorem 3.6]
Let $G$ be a finite group. Then
$$c(G) = \mathrm{min} \Big{\{} \xi(1) + m(\xi): \xi = \sum_{i \in I} \Psi_{i}, K_{I} = 1,  I \subseteq {\{1, \dots , r\}}, K_{J} \neq 1 \,\,\mathrm{if}\,\, J \subset I \Big{\}},$$
$$q(G) = \mathrm{min} \Big{\{}  \xi(1) + m(\xi): \xi = \sum_{i \in I} {m_{i}\Psi_{i}}, K_{I} = 1, I \subseteq {\{1, \dots , r\}}, K_{J} \neq 1\,\,\mathrm{if}\,\, J \subset I \Big{\}}.$$
\end{thm}
\begin{cor}[\cite{Behravesh7}, Theorem 3.7]
Let $\chi \in \mathrm{Irr}(G)$, the set of all irreducible complex characters of $G$. Then $\sum_{\alpha \in \Gamma(\chi)} {\chi^{\alpha}}$ is a rational valued character of $G$. Moreover $c(\chi)$ is a non-negative rational valued character of $G$ and $c(\chi)(1) = d(\chi) + m(\chi)$.
\end{cor}
\begin{lem}[\cite{Behravesh7}, Corollary 3.11]\label{cor1}
Let $G$ be a finite group with a unique minimal normal subgroup. Then
\begin{enumerate}
\item[(1)] $c(G) = \textit{min} \Big\{c(\chi)(1): \chi ~\textit{is a faithful irreducible complex character of G} \Big\}$;
\item[(2)] $q(G) = \textit{min} \Big\{m_{\mathbb{Q}}(\chi)c(\chi)(1): \chi~ \textit{is a faithful irreducible complex character of G} \Big\}$.
\end{enumerate}
\end{lem}
\begin{cor}
Let $G$ be a finite group with a unique minimal normal subgroup. If the Schur index of each irreducible complex character with respect to $\mathbb{Q}$ is equal to $1$, then $q(G) = c(G)$.
\end{cor}
The following result is known, but for the sake of completeness, we
include a proof.
\begin{lem}
Let $G = \langle{a}\rangle \cong C_{p^{n}}$. For $m \geq 1$, let $\Phi_{m}(X)$ denote the $m$-th cyclotomic polynomial and $\zeta_{m}$ a primitive $m$-th root of unity. Then we have the following description of inequivalent irreducible $\mathbb{Q}$-representations of $G$.
\begin{itemize}
\item[(i)] $\displaystyle X^{p^{n}} - 1 = \prod^{n}_{i = 0}\Phi_{p^{i}}(X)$ is the decomposition into monic irreducible polynomials over $\mathbb{Q}$.
\item[(ii)] The set of isomorphism types of irreducible $\mathbb{Q}$-representations of $G$ is in a bijective correspondence with the set of monic irreducible factors $\{\Phi_{1}(X), \Phi_{p}(X), \dots, \Phi_{p^{n}}(X)\}$ of $X^{p^{n}} - 1$ over $\mathbb{Q}$.
\item[(iii)] If $\rho_{0}, \rho_{1}, \dots, \rho_{n}$ are $n + 1$ inequivalent irreducible representations of $G$ over $\mathbb{Q}$, then a matrix representation of $\rho_{i}$ is defined by
$$\rho_{i}(a) = C_{\Phi_{p^i}(X)}, i = 0,1, \dots, n,$$
where $C_{\Phi_{p^i}(X)}$ denotes the companion matrix of $\Phi_{p^i}(X)$.
\item[(iv)] $G$ has a unique faithful irreducible representation over $\mathbb{Q}$. If $\chi$ is the unique faithful irreducible character of $G$ over $\mathbb{Q}$ then$$
\chi(a^i) =
  \begin{cases}
    -p^{n-1} & \quad \text{ if } (i,p^n)=p^{n-1}\\
    p^{n-1}(p-1) & \quad \text{ if } i=0\\
    0 & \quad \text{otherwise}.
  \end{cases}
$$
\end{itemize}
\end{lem}
\begin{proof}
The proof of $(i)$ follows from the well known result by Gauss that cyclotomic polynomials are irreducible over $\mathbb{Q}$.

Now we prove $(ii)$. Consider the map $\psi : \mathbb{Q}[X] \rightarrow \mathbb{Q}[G]$, and is defined by: $f(X) \mapsto f(a)$. Since $\psi$ is a ring epimorphism, $\mathbb{Q}[G] \cong \mathbb{Q}[X]/\langle{Ker(\psi)}\rangle$. So, we have
$$\displaystyle \mathbb{Q}[G] \cong \frac{\mathbb{Q}[X]}{\langle{X^{p^{n}} - 1}\rangle} = \frac{ \mathbb{Q}[X]}{\displaystyle \langle{\prod^{n}_{i = 0}\Phi_{p^{i}}(X)}\rangle}.$$ From the Chinese remainder theorem, we can write
$$\displaystyle \mathbb{Q}[G] \cong \displaystyle \bigoplus^{n}_{i = 0} \frac{ \mathbb{Q}[X]}{\langle{\Phi_{p^{i}}(X)}\rangle}.$$
Note that under the above isomorphism, $a$ is mapped to $(X + \langle{\Phi_{1}(X)}\rangle, X + \langle{\Phi_{p}(X)}\rangle, \dots,$ $ X + \langle{\Phi_{p^n}(X)}\rangle)$. As $\Phi_{p^{i}}(X)$ is an irreducible polynomial over $\mathbb{Q}$, then $\mathbb{Q}[X]/\langle{\Phi_{p^{i}}(X)}\rangle$ is a field. Hence $\mathbb{Q}[G]$ has $n + 1$ simple components $\mathbb{Q}, \mathbb{Q}(\zeta_{p}), \dots, \mathbb{Q}(\zeta_{p^n})$, as a consequence of that $G$ has $n + 1$ irreducible representations over $\mathbb{Q}$, and also they have a bijective correspondence with the set of monic irreducible factors $\{\Phi_{1}(X), \Phi_{p}(X),$ $\dots, \Phi_{p^{n}}(X)\}$ of $X^{p^{n}} - 1$ over $\mathbb{Q}$. This completes the proof.

Let $\rho_{i}$ be the irreducible representation corresponding to the simple component $\displaystyle \frac{\mathbb{Q}[X]}{\langle{\Phi_{p^{i}}(X)}\rangle}$. We consider $\displaystyle \mathbb{Q}[X]/{\langle{\Phi_{p^{i}}(X)}\rangle}$ as the representation space of $\rho_{i}$. Then $\rho_{i}$ is defined explicitly as follows.
$$\displaystyle \rho_{i}(a): \frac{ \mathbb{Q}[X]}{\langle{\Phi_{p^{i}}(X)}\rangle} \rightarrow \frac{ \mathbb{Q}[X]}{\langle{\Phi_{p^{i}}(X)}\rangle}~ \textrm{acts~by~ multiplication~of~} [X] = X + \langle{\Phi_{p^{i}}(X)}\rangle.$$ Note that the degree of the representation $\rho_{i}$ is equal to the degree of $\Phi_{p^{i}}(X)$, that is, $\phi(p^i)$. We can take $\displaystyle [1], [X], \dots, [X^{\phi(p^i)-1}]$ to be an ordered basis of $\frac{ \mathbb{Q}[X]}{\langle{\Phi_{p^{i}}(X)}\rangle}$, and with respect to that basis the matrix for $\rho_{i}(a)$ is  $C_{\Phi_{p^i}(X)},$
where $C_{\Phi_{p^i}(X)}$ denotes the companion matrix of $\Phi_{p^i}(X)$. This completes the proof of $(iii)$.

It is easy to see that $G$ has $\phi(p^n)$ many faithful irreducible
complex representations, and they are all Galois conjugates with
respect to $\mathbb{Q}$. Their direct sum gives the unique faithful
irreducible representation of $G$ over $\mathbb{Q}$. The values of
the character $\chi$ corresponding to the unique faithful
irreducible $\mathbb{Q}$-representation can be obtained from
$(iii)$, by taking traces of the various powers of the companion
matrix $C_{\Phi_{p^n}(X)}$. This completes the proof of the theorem.
\end{proof}
\vskip 5mm

\section{Wigner-Mackey Method of Little Groups}
\vskip 5mm

\noindent The groups we study are holomorphs that are semi-direct
products of a normal abelian subgroup by a subgroup. To find the set
of irreducible complex representations of such a group, we describe
the Wigner-Mackey method of little groups.

Let $G$ be a semidirect product of a normal abelian subgroup $N$ by
a subgroup $H$. Since $N$ is abelian, the set of all inequivalent
irreducible complex representations of $N$ are $1$-dimensional. Let
$\widehat{N}$ denote the set of all complex irreducible
representations of $N$.\\ Now $G$ acts on $N$ by conjugation, and
this action induces an action on $\widehat{N}$ as follows:
$$\chi \in \widehat{N}, g \in G,~ \textrm{for~all}~a \in N, \textrm{we~have}~ \chi^{g}(a) = \chi(g^{-1}ag).$$
By the definition of conjugate representations, $\chi^{g}$ is the same as conjugate representation of $\chi$ by $g \in G$. Let $I_{\chi}$ be the stabiliser subgroup of $\chi$ in $G$. Note that since $N$ is abelian, we have that $N \leq I_{\chi}$. Let $H_{\chi}:= I_{\chi} \cap H$. Then $I_{\chi}$ is a semidirect product of $N$ by $H_{\chi}$. It can be shown that any $\chi \in \widehat{N}$ can be extended to a homomorphism from $I_{\chi}$ to $\mathbb{C}^{*}$, in such a way that this extended homomorphism is the trivial map when restricted to $H_{\chi}$. So we can regard $\chi$ as an element of $\widehat{I_{\chi}}$, the set of all one dimensional complex representations of $I_{\chi}$.

Further if $\rho$ is a complex representation of $H_{\chi}$ and the canonical projection of $I_{\chi}$ on $H_{\chi}$ is composed with $\rho$, then we get a representation of $I_{\chi}$. Thus we can form the tensor product representation $\chi \otimes \rho$ of $I_{\chi}$ and deg$(\chi\otimes \rho) = $ deg $\rho$.
\begin{thm}(\cite{Serre}, Proposition 2.5)
Let $G$ be a finite group which is a semi-direct product of an abelian group $N$ by a subgroup $H$. Let $\mathcal{O}_{1}, \mathcal{O}_{2}, \dots, \mathcal{O}_{t}$ be the distinct orbits under the action of $G$ on $\widehat{N}$ and let $\chi_{j}$ be a representative of the orbit $O_{j}$. Let $I_{j}$ denote the stabiliser of $\chi_{j} $ and let $H_j:= I_j\cap H$. For any irreducible representation of $H_{j}$, let $\theta_{j, \rho}$ denote the representation of $G$ induced from the irreducible representation $\chi_{j} \otimes \rho$ of $I_{j}$. Then
\begin{enumerate}
\item $\theta_{j, \rho}$ is irreducible.
\item If $\theta_{j, \rho}$ and $\theta_{j^{'}, \rho^{'}}$ are isomorphic, then $j = j^{'}$ and $\rho$ is isomorphic to $\rho^{'}$.
\item Every irreducible representation of $G$ is isomorphic to one of the $\theta_{j, \rho}$.
\end{enumerate}
\end{thm}
\vskip 5mm

\section{A Presentation of Hol$(C_{p^{n}})$, $p$ a Prime}
\vskip 5mm

In this section, we describe presentations of Hol$(C_{p^{n}})$ which
will be useful to us;  we consider the cases of $p$ odd and $p =2$
separately. We first note:

\begin{remark}

\begin{itemize}
\item[(1)] For $p$ odd, $ \textrm{Hol}(C_{p^{n}}) \cong C_{p^{n}} \rtimes C_{p^{n-1}(p-1)}$ is a split metacyclic group of order $p^{2n-1}(p-1)$.
\item[(2)] For $n \geq 3$, $ \textrm{Hol}(C_{2^{n}}) \cong C_{2^{n}} \rtimes (C_{2^{n-2}} \times C_{2})$ is a metabelian $2$-group (which
is not metacyclic) of order $2^{2n - 1}$.
\end{itemize}
\end{remark}

Let $p$ be an odd prime and $C_{p^{n}} = \langle{x}\rangle$. Then
$\textrm{Aut}(C_{p^{n}}) = \langle{f_{r} | f_{r}: x \mapsto
x^{r}}\rangle,$ where $r$ is coprime to $p^{n}$. It is well known
that Aut$(C_{p^{n}}) \cong ({\mathbb{Z}/p^{n}\mathbb{Z}})^{*}$, the
group of units of the ring $\mathbb{Z}/p^{n}\mathbb{Z}$ and
Aut$(C_{p^{n}}) \cong A \times B$, where $A = \langle {\tau: x
\mapsto x^{s}}\rangle \cong C_{p-1}$, $s$ coprime to $p-1$ and $B =
\langle{\sigma: x \mapsto x^{1+p}}\rangle \cong C_{p^{n-1}}$.

Then, for odd primes $p$, $\textrm{Hol}(C_{p^{n}})$ has a
presentation:
$$\textrm{Hol}(C_{p^{n}}) = \langle{a, b\,|\, a^{p^{n}} = b^{p^{n-1}(p-1)} = 1, b^{-1}ab = a^r}\rangle,$$
where $p^{n-1}(p-1)$ is the order of $r$ modulo $p^{n}$.

Let $p =2$ and $C_{2^{n}} = \langle{x}\rangle$. Then
$\textrm{Aut}(C_{2^{n}}) = \langle{f_{r} \,|\, f_{r}: x \mapsto
x^{r}}\rangle,$ where $r$ is coprime to $2^{n}$.\\ For $n \geq 3$,
Aut$(C_{2^{n}}) \cong ({\mathbb{Z}/2^{n}\mathbb{Z}})^{*}$, the group
of units of the ring $\mathbb{Z}/2^{n}\mathbb{Z}$ and
Aut$(C_{2^{n}}) \cong A \times B$, where $A = \langle{\sigma: x
\mapsto x^{5}}\rangle \cong C_{2^{n-2}}$ and $B = \langle {\tau: x
\mapsto x^{-1}}\rangle \cong C_{2}$.

Then, for $n \geq 3$, $ \textrm{Hol}(C_{2^{n}})$ has a presentation:
$$\textrm{Hol}(C_{2^{n}}) = \langle{a, b, c \,|\, a^{2^{n}} = b^{2^{n-2}} = c^{2} = 1, b^{-1}ab = a^{5}, c^{-1}ac = a^{-1}, c^{-1}bc = b}\rangle.$$
\vskip 5mm

\section{Complex Representations of Hol$(C_{p^{n}})$, $p$ Odd}
\vskip 5mm

\noindent In this section, we find the set of all inequivalent
irreducible representations of Hol$(C_{p^{n}})$, $p$ odd, over
$\mathbb{C}$.

Let $G = $ Hol$(C_{p^{n}})$, $p$ odd. As we have seen before $G$ has
a presentation:
$$G=\langle a,b\,\,|\,\, a^{p^n}=1,\,\,\,\, b^{p^{n-1}(p-1)}=1,
\,\,\,\,b^{-1}ab = a^r\rangle,$$
where $r$ has multiplicative order $p^{n-1}(p-1)$ modulo $p^{n}$.

Let $N = \langle{a}\rangle$. Let $\widehat{N}$ denote the set of all irreducible complex representations of $N$, and which are given by
$$\widetilde \chi_{i}(a) = \zeta^{i}, i = 0,1 ,2, \dots, p^n -1,$$
where $\zeta$ is a primitive $p^n$-th root of unity.\\
Now, $G$ acts on $\widehat{N}$ by conjugation, with the subgroup $N$
acting trivially. So, this induces an action of $G/N$ on
$\widehat{N}$. There are $n + 1$ orbits, say, $\mathcal{O}_0,
\mathcal{O}_1, \dots, \mathcal{O}_n$ which are described as follows.

\begin{align*}
\mathcal{O}_0 :=  & \mathrm{the~orbit~of}~\widetilde \chi_{0}.
\end{align*}

\begin{align*}
{\rm For}~1 \leq k \leq n, \mathcal{O}_k := &
\mathrm{the~orbit~of}~\widetilde \chi_{p^{n-k}}.
\end{align*}

Explicitly,
\begin{align*}
\mathcal{O}_0 = & \{\widetilde \chi_{0}\}.
\end{align*}
and
\begin{align*}
\mathcal{O}_k = & \{\widetilde \chi_{p^{n-k}}, \widetilde \chi_{r.{p^{n-k}}}, \dots, \widetilde \chi_{r^{\phi(p^k)-1}.{p^{n-k}}}\}, k = 1,2, \dots, n.
\end{align*}
Further, the stabilizer subgroup of $\widetilde \chi_{0}$ is $G$
itself and the stabilizer subgroup of $\widetilde \chi_{p^{n-k}}$
(for $1 \leq k \leq n$) is $\langle{a,
b^{p^{(n-1)-(n-k)}(p-1)}}\rangle.$

Corresponding to the the orbit $\mathcal{O}_0$, the irreducible
complex representations of $G$ are given by
$$\widetilde \rho_{0, \omega}(a) = 1, \widetilde \rho_{0, \omega}(b) = \omega,$$
where $\omega$ is any $p^{n-1}(p-1)$-th root of unity. In this
manner, we obtain $\phi(p^{n})$ degree $1$ representations of $G$;
these exhaust all the degree $1$ complex representations of $G$.

Further, corresponding to each orbit $\mathcal{O}_k$ (for $1 \leq k
\leq n$), there are $o(b)/|\mathcal{O}_k|$ irreducible complex
representations of $G$, and they are of degree $\phi(p^k)$. These
irreducible complex representations are given explicitly as:
\begin{align*}
\widetilde \rho_{k,\omega}(a)=
\begin{bmatrix}
\zeta^{p^{n-k}} & 0 & 0 & \cdots & 0\\
0 & \zeta^{rp^{n-k}} & 0 &  \cdots & 0\\
0 & 0 & \zeta^{r^2p^{n-k}} & \cdots & 0\\
\vdots & \vdots & \vdots & \ddots & \\
0 & 0 & 0 &  \cdots & \zeta^{r^{\phi(p^k)-1}p^{n-k}}
\end{bmatrix},\hskip5mm
\widetilde \rho_{k,\omega}(b)=
\begin{bmatrix}
0 & 0 & 0 & \cdots & \omega\\
1 & 0 & 0 & \cdots & 0\\
0 & 1 & 0 & \cdots & 0\\
\vdots & & & \ddots & \vdots\\
0 & 0 & \cdots & 1 & 0
\end{bmatrix},
\end{align*}
where $\zeta$ is a primitive $p^{n}$-th root of unity and $\omega$ is any $o(b)/|\mathcal{O}_k| = p^{n-k}$-th root of unity.

\noindent Noting that $\phi(p^n).1 + \sum^{n}_{k =
1}p^{n-k}|\phi(p^k)|^2 = |G|$, it follows that $\widetilde\rho_{k,
w}$'s (for $0 \leq k \leq n$) exhaust all the inequivalent irreducible complex representations of $G$.

We summarise the above discussion in the following theorem.

\begin{thm}\label{Theorem $6.0.1$}
Let $p$ be an odd prime and consider the group  $G = $
Hol$(C_{p^{n}})$ for $n \geq 1$. Consider the presentation of $G$:
$$G=\langle a,b\,\,|\,\, a^{p^n}=1,\,\,\,\, b^{p^{n-1}(p-1)}=1,
\,\,\,\,b^{-1}ab = a^r\rangle,$$ where $r$ has multiplicative order
$p^{n-1}(p-1)$ modulo $p^n$. Then, with notations $\widetilde
\rho_{{k}, \omega}$, and $\mathcal{O}_k$, $0 \leq k \leq n$ as
above, we have the following.

\begin{itemize}
\item[1.] $G$ has $\phi(p^n)$ many degree $1$ complex representations, and they are: $\widetilde \rho_{{0}, \omega}$, where $\omega$ varies over the set of $p^{n-1}(p-1)$-th roots of unity.
\item[2.] Corresponding to orbit $\mathcal{O}_k$, $k = 1,2, \dots, n$, $G$ has $p^{n-k}$ number of irreducible complex representations of degree $\phi(p^k)$ and they are: $\widetilde \rho_{{k}, \omega}$, where $\omega$ varies over the set of $p^{n-k}$-th roots of unity.
\item[3.] $G$ has $(1+p+\cdots + p^{n}) - p^{n-1}$ many inequivalent irreducible complex representations.
\item[4.] Corresponding to the orbit $\mathcal{O}_n$, there is only one irreducible complex representation, which is of degree is $\phi(p^n)$. In fact, this is the unique faithful irreducible complex representation of $G$.
\end{itemize}
\end{thm}

We briefly illustrate this theorem through an example.

\begin{example}\label{example 1}
Let $G =$ Hol$(C_{9})$. $G$ has a presentation:
$$G = \langle{x, y ~|~ a^9 = b^6 = 1, b^{-1}ab = a^2}\rangle.$$
The set of all inequivalent irreducible complex representations of $N = \langle{a}\rangle$ are given by
$$\widetilde \chi_{i}(a) = \zeta^{i}, i = 0,1 ,2, \dots, 8,$$
where $\zeta$ is a primitive $3^2$-th root of unity. $G$ is acting on the set of all inequivalent irreducible complex representations of $N$ by conjugation.
There are three orbits under this action, and they are given by
$\mathcal{O}_0 = \{\widetilde{\chi}_{0}\}$;
$\mathcal{O}_1 = \{\widetilde{\chi}_{3}, \widetilde{\chi}_{6}\}$;
$\mathcal{O}_2 = \{\widetilde{\chi}_{1}, \widetilde{\chi}_{2}, \widetilde{\chi}_{4}, \widetilde{\chi}_{8}, \widetilde{\chi}_{7}, \widetilde{\chi}_{5}\}$.

Note that $2$ has order $\phi(9)$ modulo $9$.

Corresponding to $\mathcal{O}_0$, there are six irreducible representations of degree $1$, and they are defined by
$$\widetilde \rho_{0, \zeta^{i}_{6}}: a \mapsto 1, b \mapsto \zeta^{i}_{6}\,\,\, (0 \leq i \leq 5),$$
where $\zeta_{6}$ is a primitive $6$-th root of unity.

Corresponding to $\mathcal{O}_1$, there are three irreducible representations of degree $2$, and they are defined by
$$[\widetilde \rho_{1, \omega^{i}}(a)] = \begin{bmatrix}
\omega & 0 \\
0 & \omega^{2}\\
\end{bmatrix},
[\widetilde \rho_{1, \omega^{i}}(b)] = \begin{bmatrix}
0 & \omega^{i}\\
1 & 0\\
\end{bmatrix}\,\,\, (0 \leq i \leq 2),$$
where $\omega$ is a primitive cube root of unity.

Corresponding to $\mathcal{O}_2$, there is one irreducible representation of degree $6$, and which is defined by
$$[\widetilde \rho_{2, 1}(a)] = \begin{bmatrix}
\zeta & 0  & 0 & 0 & 0 & 0\\
0 & \zeta^{2} & 0 & 0 & 0 & 0\\
0 & 0 & \zeta^{4} & 0 & 0 & 0\\
0 & 0 & 0 &\zeta^{8} & 0 & 0\\
0 & 0 & 0 & 0 & \zeta^{7} & 0\\
0 & 0 & 0 & 0 & 0 & \zeta^{5}\\
\end{bmatrix},
[\widetilde \rho_{2, 1}(b)] = \begin{bmatrix}
0 & 0  & 0 & 0 & 0 & 1\\
1 & 0 & 0 & 0 & 0 & 0\\
0 & 1 & 0 & 0 & 0 & 0\\
0 & 0 & 1 & 0 & 0 & 0\\
0 & 0 & 0 & 1 & 0 & 0\\
0 & 0 & 0 & 0 & 1 & 0\\
\end{bmatrix},$$
where $\zeta$ is a primitive $9$-th root of unity. Thus $G$ has
$1+3+3^2-3=10$ inequivalent irreducible complex representations.

\end{example}
\vskip 5mm

\section{\bf Rational Representations of Hol$(C_{p^{n}})$, $p$ Odd}
\vskip 5mm

In this section, we consider as in the previous section, the groups
Hol$(C_{p^{n}})$, $p$ odd but now we describe the set of all
inequivalent irreducible representations over $\mathbb{Q}$. We will
use information on Schur index. We use the same notations as in the previous section.

\begin{lem}\label{Lemma $7.0.1$}
The Schur index of any irreducible complex representation of Hol$(C_{p^{n}})$, $p$ odd, with respect to $\mathbb{Q}$ is equal to $1$.
\end{lem}
\begin{proof}
Let $G =$ Hol$(C_{p^{n}})$, $p$ odd.
It is clear that the Schur index of any degree $1$ complex representation of $G$ with respect to $\mathbb{Q}$ is equal to $1$. Now we show that the Schur index of any non-linear irreducible complex character of $G$ with respect to $\mathbb{Q}$ is equal to $1$.

Let $H = \langle{a, b^{p-1}}\rangle$. Note that $H$ is a normal Hall subgroup of $G$ and $G/H \cong C_{p-1}$. So by Schur-Zassenhaus theorem (see \cite{J L Alperin}, Schur-Zassenhaus Theorem, Chapter 3), $H$ has a complement in $G$.

Let $\chi$ be a non-linear irreducible complex character of $G$. As we have seen before, $\chi = \phi^G$ for some irreducible complex character $\phi$ of $H$. Since $\phi$ is an irreducible complex character of the $p$-group $H$, then $\phi(1)$ is either $1$ or a power of $p$. Then by Lemma 10.8 in \cite{MR2270898}, the Schur index of $\chi$ with respect to $\mathbb{Q}$ divides $\phi(1)$. So the Schur index of $\chi$ with respect to $\mathbb{Q}$ is either $1$ or a power of $p$.
Again as $\chi = \phi^G$  and $G/H \cong C_{p-1}$, the Schur index of $\chi$ with respect to $\mathbb{Q}$ is a divisor of $p-1$. As a consequence of that the Schur index of $\chi$ with respect to $\mathbb{Q}$ is equal to $1$. This completes the proof.
\end{proof}

To find all the inequivalent irreducible rational representations of
$G$, we divide the set $\Omega_{G}$ of all inequivalent irreducible
$\mathbb{Q}$-representations of $G$ into two parts:
\begin{itemize}
\item[(1)] The irreducible $\mathbb{Q}$-representations whose kernels contain $G^{'}$.
\item[(2)] The irreducible $\mathbb{Q}$-representations whose kernels do not contain $G^{'}$.
\end{itemize}
The irreducible $\mathbb{Q}$-representations whose kernels contain $G^{'}$ can be obtained by the lifts of the irreducible $\mathbb{Q}$-representations of $G/G^{'}$.
Moreover, they are in a bijective correspondence with the Galois conjugacy classes of complex degree $1$ representations. The irreducible $\mathbb{Q}$-representations whose kernels do contain $G^{'}$ is in a bijective correspondence with the Galois conjugacy classes of irreducible complex representations of degrees equal to $1$.

Consider the orbit $\mathcal{O}_0$. As we have seen in Section $6$, the
corresponding irreducible complex representations are $\tilde
\rho_{0,\omega}$, where $\omega$ varies over the set of all
$(p^n-p^{n-1})$-th root of unity, and moreover they are all one
dimensional. Two such representations $\widetilde \rho_{0,\omega}$
and $\widetilde \rho_{0,\omega'}$ are Galois conjugate over
$\mathbb{Q}$ if and only if the complex numbers $\omega$ and
$\omega'$ are Galois conjugate over $\mathbb{Q}$; that is, their
orders are equal. The number of $\mathbb{Q}$-irreducible
representations, which are obtained in this way, is equal to the
number of divisors of $ \phi(p^n)$, where $\phi$ is the Euler's phi
function. Moreover they are determined by the factorization of
$X^{\phi(p^n)} - 1$ into irreducible polynomials over $\mathbb{Q}$.

Consider the orbits $\mathcal{O}_k$, $k = 1,2, \dots, n$. For each orbit $\mathcal{O}_k$, and for any two
$p^{n-k}$-th roots of unity $\omega,\omega'$, the irreducible complex representations $\widetilde \rho_{k,\omega}$
and $\widetilde \rho_{k,\omega'}$ are Galois conjugate over $\mathbb{Q}$ if and only if the complex numbers
$\omega$ and $\omega'$ are Galois conjugate over $\mathbb{Q}$.
Let $\omega_{i}$ be a primitive $p^i$-th root of unity.
Then for $k = 1,2, \dots, n$ and $0\le i\le n-k$, let
$$\rho_{k,i} = \bigoplus_{\sigma} (\widetilde \rho_{k,\omega_i})^{\sigma},$$
where $\sigma$ runs over the Galois group of $\mathbb{Q}(\omega_i)$
over $\mathbb{Q}$. By Theorem \ref{Theorem $1$} and Lemma
\ref{Lemma $7.0.1$}, $\rho_{k,i}$ is irreducible over
$\mathbb{Q}$.

Thus for each $k \in \{1,2, \dots, n\}$, the $p^{n-k}$ irreducible
complex  representations $\widetilde \rho_{k,\omega}$ are
partitioned into $(n-k)+1$ Galois conjugacy classes over
$\mathbb{Q}$, the (direct) sum of irreducible complex
representations in each Galois conjugacy class gives an irreducible
rational representation of $G$. Thus the number of rational
irreducible representations of $G$ which do not contain $G'$ in
their kernel is
$$\sum_{k=1}^n (n-k+1) = n+(n-1) + \cdots + 1 =\frac{n(n+1)}{2}.$$

The above discussion can be summarised as the following theorem:

\begin{thm}\label{Theorem $7.0.2$}
Let $p$ be an odd prime and consider the group  $G = $
Hol$(C_{p^{n}})$ as before, with the presentation:
$$G=\langle a,b\,\,|\,\, a^{p^n}=1,\,\,\,\, b^{p^{n-1}(p-1)}=1,
\,\,\,\,b^{-1}ab = a^r\rangle,$$ where $r$ has multiplicative order
$p^{n-1}(p-1)$ modulo $p^n$. For $m \geq 1$, let $\Phi_{m}(X)$ denote $m$-th cyclotomic polynomial. Then we have the following results on
its irreducible representations over $\mathbb{Q}$:

\begin{enumerate}
\item[1.]  The set of all inequivalent irreducible $\mathbb{Q}$-representations of $G$ whose kernels contain $G^{'}$ is in a bijective correspondence with the Galois
conjugacy classes of complex degree $1$ representations, and also in
a bijective correspondence with monic irreducible polynomials
$\Phi_{d}(X)$, as $d$ varies over divisors of $\phi(p^n)$.
\item[2.] The set of all inequivalent irreducible $\mathbb{Q}$-representations of $G$ whose kernels do not contain $G^{'}$ is
in a bijective correspondence with the set of all monic irreducible
polynomials $\Phi_{1}(X), \Phi_{p}(X), \dots, \Phi_{p^{n-k}}(X)$,
where $k$ runs over $\{ 1, 2, \dots, n\}$.
\item[3.]  The number of inequivalent irreducible
$\mathbb{Q}$-representations of $G$ is equal to $d +
\frac{n(n+1)}{2}$, where $d$ is the number of divisors of
$\phi(p^n)$.

\end{enumerate}
\end{thm}
\vskip 5mm

\section{Rational Wedderburn decomposition of Hol$(C_{p^{n}})$, $p$
odd} \vskip 5mm

As we have information on the Schur indices form the Lemma
\ref{Lemma $7.0.1$}, and the description of irreducible
representations over $\mathbb{Q}$ from the previous theorem, we may
easily find the Wedderburn decomposition of the group algebra of
Hol$(C_{p^{n}})$, $p$ odd, over $\mathbb{Q}$. With the same
notations as in the previous section, we prove:

\begin{thm}\label{Theorem $8.0.1$}
Let $G =$ Hol$(C_{p^{n}})$, $p$ odd. For $m \geq 1$, let $\zeta_{m}$ denote a primitive $m$-th root of unity. The Wedderburn decomposition of $\mathbb{Q}[G]$ is
$$\mathbb{Q}[G] = \bigoplus_{d | \phi(p^n)}\mathbb{Q}(\zeta_{d}) \bigoplus^{n-k}_{l = 0}(\bigoplus^{n}_{k = 1} M_{\phi(p^k)}(\mathbb{Q}(\zeta_{p^{l}}))).$$
\end{thm}
\begin{proof}
Let $\rho: G \rightarrow \textrm{GL}(V)$ be an irreducible $\mathbb{Q}$-representation of $G.$ Let the minimal $2$-sided ideal of $\mathbb{Q}[G]$ corresponding to
$\rho$ be abstractly isomorphic to $M_n(D),$ for suitable $n$, and a division ring $D$. Let $Z$ be the center of $D$. By Lemma \ref{Lemma $7.0.1$}, the Schur index of any irreducible complex representation of $G$ with respect to $\mathbb{Q}$ is equal to $1$. Therefore the Schur index of $D$ (or $M_n(D)$) over $\mathbb{Q}$ is equal to $1$, and consequently $D = Z$. By Theorem \ref{Theorem $1$}, $\rho$ is completely reducible over $\mathbb{C}$ into a certain number inequivalent irreducible $\mathbb{C}$-representations $\widetilde{\rho} = \widetilde{\rho_{1}}, \widetilde{\rho_{2}}, \dots, \widetilde{\rho_{\delta}}$ with the same multiplicity $1$, and they are Galois conjugates with respect to $\mathbb{Q}$.
Let $\widetilde{\chi}_{i}$ is the character of $\widetilde{\rho_{i}}$. Then by Theorem $3$ in \cite{MR0122892}, $Z \cong F(\widetilde \chi_{i})$ and $\delta = [F(\widetilde \chi_{i}), F]$. Again as the Schur index of any irreducible complex representation of $G$ with respect to $\mathbb{Q}$ is equal to $1$, then $n$ is equal to the common degree of the representations $\widetilde{\rho} = \widetilde{\rho_{1}}, \widetilde{\rho_{2}}, \dots, \widetilde{\rho_{\delta}}$.

If $\rho$ is the irreducible $\mathbb{Q}$-representation whose kernel contain $G^{'}$,
then the minimal $2$-sided ideal of $\mathbb{Q}[G]$ corresponding to
$\rho$ be abstractly isomorphic to $\mathbb{Q}(\zeta_{d})$, where $d$ is a divisor of $\phi(p^n)$.

If $\rho$ is an irreducible $\mathbb{Q}$-representation whose kernel does not contain $G^{'}$, then for some $k \in \{1,2, \dots, n\}$ and $0\le l\le n-k$,
$$\rho = \rho_{k,l} = \bigoplus_{\sigma} (\widetilde \rho_{k,\zeta_{p^l}})^{\sigma},$$
where $\zeta_{p^l}$ is a primitive $p^l$-th root of unity and
$\sigma$ runs over the Galois group of $\mathbb{Q}(\zeta_{p^l})$
over $\mathbb{Q}$. Note that the degree of $\widetilde
\rho_{k,\zeta_{p^{l}}}$ is $\phi(p^k)$ and the character field of
$\widetilde \rho_{k,\zeta_{p^{l}}}$ over $\mathbb{Q}$ is
$\mathbb{Q}(\zeta_{p^l})$. Therefore the minimal $2$-sided ideal of
$\mathbb{Q}[G]$ corresponding to $\rho$ is abstractly isomorphic to
$M_{\phi(p^k)}(\mathbb{Q}(\zeta_{p^l}))$. This completes the proof.

\end{proof}

Let us illustrate Theorem 7.0.2 and the previous theorem through an
example.

\begin{example}
Let $G =$ \textrm{Hol}$(C_{9})$. $G$ has a presentation of the form:
$$G = \langle{a, b ~|~ a^9 = b^6 = 1, b^{-1}ab = a^2}\rangle.$$
We have seen in Example \ref{example 1} that corresponding to the
orbit {$\mathcal{O}_0 = \{1\}$}, there are six irreducible
representations of degree $1$. \\
Over $\mathbb{Q}$,
$$X^{6} - 1 =
\phi_{1}(X)\phi_{2}(X)\phi_{3}(X)\phi_{6}(X) = (X -1)(X +1)(X^{2} +
X +1)(X^{2} - X +1)$$ is the factorization into monic irreducible
polynomials. Therefore, there are four $\mathbb{Q}$-irreducible
representations whose kernels contain $G^{'}$, and they can be
realized as
 $$\rho_{0, 1}(a) \mapsto 1,
 \rho_{0, 1}(b) \mapsto 1;$$
$$\rho_{0, -1}(a) \mapsto 1,
\rho_{0, -1}(b) \mapsto -1;$$
$$[\rho_{0, \omega}(a)] = \begin{bmatrix}
 1 & 0 \\
 0 & 1\\
 \end{bmatrix},
 [\rho_{0, \omega}(b)] = \begin{bmatrix}
 0 & -1\\
 1 & -1\\
 \end{bmatrix};$$
$$[\rho_{0, \zeta_{6}}(a)] = \begin{bmatrix}
1 & 0 \\
0 & 1\\
\end{bmatrix},
[\rho_{0, \zeta_{6}}(b)] = \begin{bmatrix}
0 & -1\\
1 & 1\\
\end{bmatrix},$$
where $\omega$ is a primitive cube root of unity and $\zeta_{6}$
is a primitive $6$-th root of unity.\\
In general, we get companion matrices of the irreducible polynomial
factors.

We have seen in Example \ref{example 1} that corresponding to the
orbit $\mathcal{O}_1$, there are three irreducible representations
of degree $2$. \\
The representation $\widetilde \rho_{1,1}$ is realizable over
$\mathbb{Q}$, and therefore irreducible. \\
The representations $\widetilde \rho_{1, \omega}$ and $\widetilde
\rho_{1, \omega^2}$ are Galois conjugate over $\mathbb{Q}$, and hence $\widetilde \rho_{1, \omega}\oplus \widetilde \rho_{1, \omega^2}$ irreducible over $\mathbb{Q}$.\\
The representation $\widetilde \rho_{2,1}$ is realizable over
$\mathbb{Q}$, and hence irreducible over $\mathbb{Q}$.\\

The Wedderburn decomposition of $\mathbb{Q}[G]$ is:
$$\mathbb{Q}[Hol(C_9)] = \bigoplus_{d | 6}\mathbb{Q}(\zeta_{d}) \bigoplus M_{2}(\mathbb{Q}) \bigoplus M_{2}(\mathbb{Q}(\omega)) \bigoplus M_{6}(\mathbb{Q}).$$
\end{example}

\vskip 5mm

\section{Representations of Hol$(C_{p^{n}})$, $p = 2$}
\vskip 5mm

In this section, we consider the case $p=2$. That is, we find the
set of all inequivalent irreducible representations of
Hol$(C_{p^{n}})$, $p = 2$, over $\mathbb{C}$ and $\mathbb{Q}$. We
will see that the end results are analogous to the odd case but the
intermediary descriptions are different and more complicated as the
group is not metacyclic when $n>2$.

\subsection{Complex Irreducible Represntations of Hol$(C_{p^{n}})$, $p = 2$}

We have seen that for $n \geq 3$, $ \textrm{Hol}(C_{2^{n}})$ has a presentation of the form:
$$G = \textrm{Hol}(C_{2^{n}}) = \langle{a, b, c \,|\, a^{2^{n}} = b^{2^{n-2}} = c^{2} = 1, b^{-1}ab = a^{5}, c^{-1}ac = a^{-1}, c^{-1}bc = b}\rangle.$$

Let $N = \langle{a}\rangle \cong C_{2^{n}}$. Let $\widehat{N}$ denote the set of all inequivalent irreducible representations of $N$ over $\mathbb{C}$. The set of all irreducible representations of $N$ are given by
$$\widetilde \chi_{i}(a) = \zeta^{i}, i = 0,1 ,2, \dots, 2^n -1,$$
where $\zeta$ is a primitive $2^n$-th root of unity.
Now $G$ is acting on $\widehat{N}$ by conjugation, $N$ is acting trivially, which induces an action of $G/N$ on $\widehat{N}$.
There are $n + 1$ distinct orbits. Let $\mathcal{O}_{0}$ denote the orbit of $\widetilde{\chi}_{0}$. Let $\mathcal{O}_{k}$ denote the orbit of $\widetilde{\chi}_{2^{n-k}}$, $k = 1,2, \dots, n$.
In fact, $\mathcal{O}_{0}, \mathcal{O}_{1}, \dots, \mathcal{O}_{n}$ are all the distinct $n + 1$ orbits. One can see that orbits $\mathcal{O}_{0}$ and $\mathcal{O}_{1}$ are singleton sets, and so the corresponding stabilizer subgroup in $G$ is $G = \langle{a, b, c}\rangle$ itself. Let $\textrm{Stab}(\widetilde{\chi}_{2^{n-k}})$ denote the stabilizer subgroup of $\widetilde{\chi}_{2^{n-k}}$, $k = 2, \dots, n$. One can see that
$$\textrm{Stab}(\widetilde{\chi}_{2^{n-k}}) = \langle{a, b^{2^{(n-2)-(n-k)}}}\rangle, k = 2, \dots, n.$$
So for each $k = 2, \dots, n$, $\textrm{Stab}(\widetilde{\chi}_{2^{n-k}})$ is of order $2^{2n-k}$. Corresponding to the orbits $\mathcal{O}_{0}$ and $\mathcal{O}_{1}$, there are $2^{n-1}$ many irreducible complex representations of degree $1$.
Corresponding to each orbit $\mathcal{O}_k, k = 2, \dots, n$, there are
$\frac{\textrm{Stab}(\widetilde{\chi}_{2^{n-k}})}{|N|} = 2^{n-k}$ many distinct
irreducible complex representations of $G$, and of degree $|\mathcal{O}_k| = \phi(2^k) = 2^{k-1}$.

Corresponding to the orbit $\mathcal{O}_0$,
the irreducible complex representations of $G$ are given by
$$\widetilde \rho_{0,\omega, 1}(a)= 1, \widetilde \rho_{0,\omega, 1}(b)= \omega, \widetilde \rho_{0,\omega, 1}(c)= 1$$
and
$$\widetilde \rho_{0,\omega, -1}(a)= 1, \widetilde \rho_{0,\omega, -1}(b)= \omega, \widetilde \rho_{0,\omega, -1}(c)= -1,$$
where $\omega$ is any $2^{n-2}$-th root of unity. In this way, we obtain $2^{n-1}$ distinct irreducible complex representations of degree $1$.

Corresponding to the orbit $\mathcal{O}_1$, the irreducible complex representation of $G$ are given by
$$\widetilde \rho_{1,\omega, 1}(a)= -1, \widetilde \rho_{1,\omega, 1}(b)= \omega, \widetilde \rho_{1,\omega, 1}(c)= 1$$
and
$$\widetilde \rho_{1,\omega, -1}(a)= -1, \widetilde \rho_{1,\omega, -1}(b)= \omega, \widetilde \rho_{1,\omega, -1}(c)= -1,$$
where $\omega$ is a primitive $2^{n-2}$-th root of unity. In this way, we obtain $2^{n-1}$ distinct irreducible complex representations of degree $1$.

Thus we have obtained $2^{n}$ number of degree $1$ irreducible complex representations of $G$; these exhaust all the degree $1$ complex representations of $G$.

Corresponding to the orbit $\mathcal{O}_k, k = 2, \dots, n$,
the irreducible complex representation of $G$ are given explicitly as:
\begin{align*}
\tiny
\widetilde \rho_{k,\omega}(a)=
\begin{bmatrix}
\zeta^{2^{n-k}} & 0  & \cdots & 0 & \vline & 0 & 0 & \cdots & 0 \\
0 & \zeta^{5.2^{n-k}} &  \cdots & 0 & \vline & 0 & 0 &   \cdots & 0\\
\vdots & \vdots & \ddots & \vdots & \vline & \vdots & \vdots \ddots & \vdots \\
0 & 0 &   \cdots & \zeta^{5^{2^{(n-2)-(n-k)}-1}.2^{n-k}}& \vline & 0 & 0 &  \cdots & 0 \\
\hline
0 & 0 &   \cdots & 0 & \vline & \zeta^{-2^{n-k}} & 0 &  \cdots & 0 \\
0 & 0 &   \cdots & 0 & \vline & 0 & \zeta^{-5.2^{n-k}} &  \cdots & 0 \\
\vdots & \vdots & \ddots & \vdots & \vline & \vdots& \vdots & \vdots \ddots & \vdots\\
0 & 0 & \cdots & 0 & \vline & 0 & 0 &  \cdots & \zeta^{-5^{2^{(n-2)-(n-k)}-1}.2^{n-k}}\\
\end{bmatrix},
\end{align*}
\begin{align*}
\tiny
\widetilde \rho_{k,\omega}(b)=
\begin{bmatrix}
0 & 0  & \cdots & \omega & \vline & 0 & 0 & \cdots & 0 \\
1 & 0 &  \cdots & 0 & \vline & 0 & 0 &   \cdots & 0\\
\vdots & \vdots & \ddots & \vdots & \vline & \vdots & \vdots \ddots & \vdots \\
0 & 0 &   \cdots & 0 & \vline & 0 & 0 &  \cdots & 0 \\
\hline
0 & 0 &   \cdots & 0 & \vline & 0 & 0 &  \cdots & \omega \\
0 & 0 &  \cdots & 0 & \vline & 1 & 0 &   \cdots & 0\\
\vdots & \vdots & \ddots & \vdots & \vline & \vdots& \vdots & \vdots \ddots & \vdots\\
0 & 0 & \cdots & 0 & \vline & 0 & 0 &  \cdots & 0\\
\end{bmatrix},
\tiny
\widetilde \rho_{k,\omega}(c)=
\begin{bmatrix}
0 & 0  & \cdots & 0 & \vline & 1 & 0 & \cdots & 0 \\
0 & 0 &  \cdots & 0 & \vline & 0 & 1 &   \cdots & 0\\
\vdots & \vdots & \ddots & \vdots & \vline & \vdots & \vdots \ddots & \vdots \\
0 & 0 &   \cdots & 0 & \vline & 0 & 0 &  \cdots & 1 \\
\hline
1 & 0 &   \cdots & 0 & \vline & 0 & 0 &  \cdots & 0 \\
0 & 1 &   \cdots & 0 & \vline & 0 & 0 &  \cdots & 0 \\
\vdots & \vdots & \ddots & \vdots & \vline & \vdots& \vdots & \vdots \ddots & \vdots\\
0 & 0 & \cdots & 1 & \vline & 0 & 0 &  \cdots & 0\\
\end{bmatrix}.
\end{align*}
where $\zeta$ is a primitive $2^{n}$-th root of unity and $\omega$
is any $2^{n-k}$-th root of unity. \vskip 3mm

In the above discussion, we have obtained $(2^{n-1}+2^{n-1}) +
\sum_{k=2}^{n} 2^{n-k} = (1+2+\cdots + 2^{n}) - 2^{n-1}$ distinct
complex irreducible representations of $G$. Noting that

\begin{align*}
& 2^{n-1}(1) + 2^{n-1}(1)+ 2^{n-2}(2)^2+ \dots + 2(2^{n-2})^2 + (2^{n-1})^2\\
& = 2^{n}+ 2^{n} + 2^{n+1}+ \dots + 2^{2n-3} + 2^{2n-2}\\
& = 2^{n}+ 2^{n}( 1 + 2 + 2^{2} + \dots + 2^{n-2})\\
& = 2^{n}+ 2^{n}(2^{n-1} - 1) = 2^{2n-1} = |G|.
\end{align*}
It follows that these exhaust all the inequivalent irreducible
complex representations of $G$.

Summarising the above discussion, we obtain the following theorem.

\begin{thm}\label{Theorem $9.1.1$}
Consider the group $G= \textrm{Hol}(C_{2^{n}})$ which has a
presentation:
$$G = \textrm{Hol}(C_{2^{n}}) = \langle{a, b, c \,|\, a^{2^{n}} = b^{2^{n-2}} = c^{2} = 1, b^{-1}ab = a^{5}, c^{-1}ac = a^{-1}, c^{-1}bc = b}\rangle$$
for $n \geq 3$. Then, we have the following.

\begin{itemize}
\item[1.] $G$ has $2\phi(2^n) (= 2^{n})$ many degree $1$ complex representations, and they are: $\widetilde \rho_{{0}, \omega, \pm 1}$ and $\widetilde \rho_{{1}, \omega, \pm 1}$, where $\omega$ varies over the set of $2^{n-2}$-th roots of unity.
\item[2.] Corresponding to orbit $\mathcal{O}_k$, $k = 2, \dots, n$, $G$ has $2^{n-k}$ number of irreducible complex representations of degree $\phi(2^k)$ and they are: $\widetilde \rho_{{k}, \omega}$, where $\omega$ varies over the set of $2^{n-k}$-th roots of unity.
\item[3.] $G$ has $(1+2+\cdots + 2^{n}) - 2^{n-1}$ many inequivalent irreducible complex representations.
\item[4.] Corresponding to the orbit $\mathcal{O}_n$, there is only one irreducible complex representation, which is of degree is $\phi(2^n)$. In fact, this is the unique faithful irreducible complex representation of $G$.
\end{itemize}
\end{thm}

As mentioned in the beginning of this section, the form of the above
theorem is exactly similar to that of the analogous theorem when $p$
is odd, but the descriptions of the complex irreducible
representations was more involved in the case $p=2$.

Here is an example illustrating the Theorem \ref{Theorem $9.1.1$}.
We use the notations as in the above.
\begin{example}\label{example 2}
Let $G = \mathrm{Hol}(C_{2^5})$; it has a presentation
\begin{align*}
\langle{a, b, c\,|\, a^{2^5} = b^{2^3} = c^2 = 1, b^{-1}ab = a^5, c^{-1}ac = c^{-1}, c^{-1}bc = b}\rangle.
\end{align*}
Let $N = C_{2^5} = \langle{a}\rangle$. Let $\widehat{N}$ denote the set of all inequivalent irreducible representations of $N$ over the field of complex numbers. The set of all inequivalent irreducible complex representations of $N$ are given by
$$\widetilde{\chi_{i}}(a) = \zeta^{2^5 - i}, i = 0,1,2, \dots, 2^5-1.$$
$G$ is acting on $\widehat{N}$ by conjugation, and $N$ is acting trivially. This induces an action of $G/N$ on $\widehat{N}$.

Let $\mathcal{O}_{0}$ denote the orbit of $\widetilde \chi_{0}$ and $\mathcal{O}_{k}$ denote the orbit of $\widetilde \chi_{{2^{5-k}}}, k = 1,2, 3, 4, 5$.

Corresponding to the orbit $\mathcal{O}_0$,
the complex irreducible representations of $G$ are given by
$$\widetilde \rho_{0,\omega, 1}(a)= 1, \widetilde \rho_{0,\omega, 1}(b)= \omega, \widetilde \rho_{0,\omega, 1}(c)= 1$$
and
$$\widetilde \rho_{0,\omega, -1}(a)= 1, \widetilde \rho_{0,\omega, -1}(b)= \omega, \widetilde \rho_{0,\omega, -1}(c)= 1,$$
where $\omega$ is any $2^{3}$-th root of unity. In this way, we obtain $2^{4}$ distinct irreducible complex representations of degree $1$.

Corresponding to the orbit $\mathcal{O}_1$, the complex irreducible representation of $G$ are given by
$$\widetilde \rho_{1,\omega, 1}(a)= 1, \widetilde \rho_{1,\omega, 1}(b)= \omega, \widetilde \rho_{1,\omega, 1}(c)= 1$$
and
$$\widetilde \rho_{1,\omega, -1}(a)= 1, \widetilde \rho_{1,\omega, -1}(b)= \omega, \widetilde \rho_{1,\omega, -1}(c)= -1,$$
where $\omega$ is a primitive $2^{3}$-th root of unity. In this way, we obtain $2^{4}$ distinct irreducible complex representations of degree $1$.

Corresponding to the orbit $\mathcal{O}_k, k = 2, 3, 4, 5$,
the complex irreducible representation of $G$ are given by

\begin{align*}
\tiny
\widetilde \rho_{k,\omega}(a)= \begin{bmatrix}
\zeta^{2^{5-k}} & 0  & \cdots & 0 & \vline & 0 & 0 & \cdots & 0 \\
0 & \zeta^{5.2^{5-k}} &  \cdots & 0 & \vline & 0 & 0 &   \cdots & 0\\
\vdots & \vdots & \ddots & \vdots & \vline & \vdots & \vdots & \ddots & \vdots \\
0 & 0 &   \cdots & \zeta^{5^{2^{(5-2)-(5-k)}-1}.2^{5-k}} & \vline & 0 & 0 &  \cdots & 0 \\
\hline
0 & 0 &   \cdots & 0 & \vline & \zeta^{-2^{5-k}} & 0 &  \cdots & 0 \\
0 & 0 &   \cdots & 0 & \vline & 0 & \zeta^{-5.2^{5-k}} &  \cdots & 0 \\
\vdots & \vdots & \ddots & \vdots & \vline & \vdots& \vdots & \vdots \ddots & \vdots\\
0 & 0 & \cdots & 0 & \vline  & 0 & 0 &  \cdots & \zeta^{-5^{2^{(5-2)-(5-k)}-1}.2^{5-k}}\\
\end{bmatrix},
\end{align*}
\begin{align*}
\tiny
\widetilde \rho_{k,\omega}(b)=
\begin{bmatrix}
0 & 0  & \cdots & \omega & \vline & 0 & 0 & \cdots & 0 \\
1 & 0 &  \cdots & 0 & \vline & 0 & 0 &   \cdots & 0\\
\vdots & \vdots & \ddots & \vdots & \vline & \vdots & \vdots & \ddots & \vdots \\
0 & 0 &   \cdots & 0 & \vline & 0 & 0 &  \cdots & 0 \\
\hline
0 & 0 &   \cdots & 0 & \vline & 0 & 0 &  \cdots & \omega \\
0 & 0 &  \cdots & 0 & \vline & 1 & 0 &   \cdots & 0\\
\vdots & \vdots & \ddots & \vdots & \vline & \vdots& \vdots & \vdots \ddots & \vdots\\
0 & 0 & \cdots & 0 & \vline & 0 & 0 &  \cdots & 0\\
\end{bmatrix},
\tiny
\widetilde \rho_{k,\omega}(c)=
\begin{bmatrix}
0 & 0  & \cdots & 0 & \vline & 1 & 0 & \cdots & 0 \\
0 & 0 &  \cdots & 0 & \vline & 0 & 1 &   \cdots & 0\\
\vdots & \vdots & \ddots & \vdots & \vline & \vdots & \vdots \ddots & \vdots \\
0 & 0 &   \cdots & 0 & \vline & 0 & 0 &  \cdots & 1 \\
\hline
1 & 0 &   \cdots & 0 & \vline & 0 & 0 &  \cdots & 0 \\
0 & 1 &   \cdots & 0 & \vline & 0 & 0 &  \cdots & 0 \\
\vdots & \vdots & \ddots & \vdots & \vline & \vdots& \vdots & \vdots \ddots & \vdots\\
0 & 0 & \cdots & 1 & \vline & 0 & 0 &  \cdots & 0\\
\end{bmatrix}.
\end{align*}
where $\zeta$ is a primitive $2^{5}$-th root of unity and $\omega$ is any $2^{5-k}$-th root of unity.

Thus $G$ has $(1+2+\cdots + 2^{5}) - 2^{4}$ distinct irreducible
complex  representations.
\end{example}

\vskip 5mm

\subsection{\bf Rational Representations of Hol$(C_{p^{n}})$, $p = 2$}
\vskip 5mm

Now, we turn to a description of the set of inequivalent irreducible
$\mathbb{Q}$-representations of  Hol$(C_{2^{n}})$. This is again
more complicated than the odd $p$ case but the form of the end
result is the same. We use the same notations as in the previous section.

\begin{lem}\label{Lemma $9.2.1$}
The Schur index of any irreducible complex representation of Hol$(C_{2^{n}})$ with respect to $\mathbb{Q}$ is equal to $1$.
\end{lem}
\begin{proof}
Let $G =$ Hol$(C_{2^{n}})$. We have seen that for $n \geq 3$, $ \textrm{Hol}(C_{2^{n}})$ has a presentation of the form:
$$G = \textrm{Hol}(C_{2^{n}}) = \langle{a, b, c \,|\, a^{2^{n}} = b^{2^{n-2}} = c^{2} = 1, b^{-1}ab = a^{5}, c^{-1}ac = a^{-1}, c^{-1}bc = b}\rangle.$$

Each irreducible complex representation in the orbit
$\mathcal{O}_{0}$ and $\mathcal{O}_{1}$ extends $\phi(2^n)$ distinct
ways to $G$, and they are of degree $1$. So the Schur index of each
such irreducible complex representation with respect to $\mathbb{Q}$
is equal to $1$.

As we have seen in the previous section that for each $k, (2 \leq
k \leq n)$, there are $2^{n-k}$ many irreducible complex
representations of $G$, and they are of degree $\phi(2^k) = 2^{k-1}$. Each of
them is induced from an one dimensional representation of the
subgroup $\langle a, b^{2^{(n-2)-(n-k)}}\rangle$, which can be
obtained by extension of the irreducible representation
$\widetilde{\chi}_{2^{n-k}}$ of the subgroup $\langle{a}\rangle$ to
$\langle a, b^{2^{(n-2)-(n-k)}}\rangle$. We have seen that
$\widetilde{\chi}_{2^{n-k}}$ extends $2^{n-k}$ many distinct ways to
$\langle a, b^{2^{(n-2)-(n-k)}}\rangle$. We denote all the $2^{n-k}$ extensions of $\widetilde \chi_{2^{n-k}}$ by $\widetilde{\chi}_{2^{n-k}, \omega}$, where $\omega$ runs over the set of all the $2^{n-k}$-th roots of unity.
They are defined by
$\widetilde \chi_{2^{n-k}, \omega}(a) = \zeta^{2^{n-k}},$ where
$\zeta$ is a primitive $2^n$-th root of unity, and
$\widetilde{\chi}_{2^{n-k}, \omega}(b^{2^{(n-2)-(n-k)}}) = \omega$,
where $\omega$ runs over the set of all the $2^{n-k}$-th roots of
unity. By Clifford's theorem (see \cite{MR2270898}, Theorem 6.2), $\widetilde \rho_{k, \omega} = \widetilde{\chi}_{2^{n-k},
\omega}^G$ is irreducible. It is easy to see that the character
field of each $\widetilde{\chi}_{2^{n-k}, \omega}$ over $\mathbb{Q}$
is $\mathbb{Q}(\zeta^{2^{n-k}}, \omega)$ and the character field of
$\widetilde \rho_{k, \omega}$ is equal to $\mathbb{Q}(\omega)$. If the character field
of $\widetilde \rho_{k, \omega}$ contains $\sqrt{-1}$ then by Roquette's theorem
(see \cite{MR2270898}, Theorem 10.14), the Schur index of $\widetilde
\rho_{k, \omega}$ with respect to $\mathbb{Q}$ is equal to $1$. For
each $k$, $(2 \leq k \leq n-1)$, there are two irreducible complex
representations of $G$, namely $\widetilde \rho_{k, 1}$ and $\widetilde \rho_{k, -1}$, with character field
$\mathbb{Q}$. It can be easily shown that the Schur index of
$\widetilde \rho_{k, 1}, \widetilde \rho_{k, -1}$ with respect to
$\mathbb{Q}$ is equal to $1$.

For $k = n$, there is only one irreducible complex
representation of $G$ of degree $\phi(2^n) = 2^{n-1}$, which is
induced representation of $\widetilde \chi_{1}$ to $G$. Since $N = \langle a \rangle$ has a
complement in $G$, by Lemma 10.8 in \cite{MR2270898}, the
Schur index of $\widetilde{\chi}^{G}_{1}$ with respect to
$\mathbb{Q}$ is equal to $1$. This completes the proof.
\end{proof}

Now, thanks to this Lemma, we may find all the inequivalent
irreducible rational representations of $G$ as we did in the odd $p$
case.

Once again, the set $\Omega_{G}$ of all inequivalent irreducible
$\mathbb{Q}$-representations of $G$ is comprised of two parts:

\begin{itemize}
    \item[(1)] The irreducible $\mathbb{Q}$-representations whose kernels contain $G^{'}$.
    \item[(2)] The irreducible $\mathbb{Q}$-representations whose kernels do not contain $G^{'}$.
\end{itemize}

The irreducible $\mathbb{Q}$-representations whose kernels contain
$G^{'}$ can be obtained by the lifts of the irreducible
$\mathbb{Q}$-representations of $G/G^{'}$. Moreover, they are in a
bijective correspondence with the Galois conjugacy classes of
complex degree $1$ representations. The irreducible
$\mathbb{Q}$-representations whose kernels do contain $G^{'}$ is in
a bijective correspondence with the Galois conjugacy classes of
irreducible complex representations of degrees equal to $1$.

Consider the orbit $\mathcal{O}_0$. The corresponding irreducible complex
representations are $\widetilde \rho_{0,\omega, \pm 1}$, where $\omega$ is
a $2^{n-2}$-th root of unity, are all of degree $1$. For any two
$2^{n-2}$-th roots of unity $\omega,\omega'$, two representations
$\widetilde \rho_{0,\omega, 1}$ and $\widetilde \rho_{0,\omega', 1}$ are Galois conjugate over $\mathbb{Q}$ if and only if the complex numbers $\omega$ and $\omega'$ are Galois conjugate over $\mathbb{Q}$. Similarly, two representations
$\widetilde \rho_{0,\omega, -1}$ and $\widetilde \rho_{0,\omega', -1}$ are Galois conjugate over $\mathbb{Q}$ if and only if the complex numbers $\omega$ and $\omega'$ are Galois conjugate over $\mathbb{Q}$.
So the total number of irreducible $\mathbb{Q}$-representations, which are obtained in this way, is equal to $2(n-1)$.

Consider the orbit $\mathcal{O}_1$.
The corresponding complex irreducible
representations are $\widetilde \rho_{1,\omega, \pm 1}$, where $\omega$ is
a $2^{n-2}$-th root of unity. For any two
$2^{n-2}$-th roots of unity $\omega,\omega'$, two representations
$\widetilde \rho_{1,\omega, 1}$ and $\widetilde \rho_{1,\omega', 1}$ are Galois conjugate over $\mathbb{Q}$ if and only if the complex numbers $\omega$ and $\omega'$ are Galois conjugate over $\mathbb{Q}$. Similarly, two representations
$\widetilde \rho_{1,\omega, -1}$ and $\widetilde \rho_{1,\omega', -1}$ are Galois conjugate over $\mathbb{Q}$ if and only if the complex numbers $\omega$ and $\omega'$ are Galois conjugate over $\mathbb{Q}$.
So the total number of irreducible $\mathbb{Q}$-representations, which are obtained in this way, is equal to $2(n-1)$.

Thus in this way, we obtain $4(n-1)$ many irreducible $\mathbb{Q}$-representations of $G$.

Consider the orbits $\mathcal{O}_k$, $2\le k\le n$. For each orbit $\mathcal{O}_k$ (of size $\phi(2^k)$), and for any two
$2^{n-k}$-th roots of unity $\omega,\omega'$, the irreducible complex representations $\widetilde \rho_{k,\omega}$
and $\widetilde \rho_{k,\omega'}$ are Galois conjugates over $\mathbb{Q}$ if and only if the complex numbers
$\omega$ and $\omega'$ are Galois conjugates over $\mathbb{Q}$.
Let $\omega_{j}$ denote a primitive $2^j$-th root of unity.

Then for $2\le k\le n$ and $0\le i\le n-k$, let
$$\rho_{k,i} = \bigoplus_{\sigma} (\widetilde \rho_{k,\omega_i})^{\sigma},$$
where $\sigma$ runs over the Galois group of $\mathbb{Q}(\omega_i)$ over $\mathbb{Q}$.

Clearly the character corresponding to $\rho_{k,i}$ is rational
valued. By Lemma \ref{Lemma $9.2.1$}, as the Schur index of $\widetilde \rho_{k,\omega_i}$ with
respect to $\mathbb{Q}$ is equal to $1$, then $\rho_{k,i}$ is an
irreducible representation of $G$ over $\mathbb{Q}$. Thus for each
$k$ ($2\le k\le n$), $2^{n-k}$ many irreducible complex
representations $\widetilde \rho_{k,\omega}$ are partitioned into
$(n-k)+1$ Galois conjugacy classes over $\mathbb{Q}$, the (direct)
sum of irreducible complex representations in each Galois conjugacy
class gives an irreducible rational representation of $G$.

Hence, the number of irreducible $\mathbb{Q}$-representations of $G$
which do not contain $G'$ in their kernel equals
$$\sum_{i=2}^n (n-k+1) = (n-1)+(n-2) + \cdots + 1 =\frac{n(n-1)}{2}.$$
So, the total number of inequivalent irreducible
$\mathbb{Q}$-representations of $G$ equals $4(n-1) + n(n-1)/2$.

We may summarise the above discussion in the following theorem:

\begin{thm}\label{Theorem $9.2.2$}
Consider the holomorph $G = Hol(C_{2^n})$ for $n \geq 3$. For $m \geq 1$, let $\Phi_{m}(X)$ denote $m$-th cyclotomic polynomial. Then, we
have the following description of its inequivalent, irreducible
representations over $\mathbb{Q}$:

\begin{enumerate}
\item[1.] The set of all inequivalent irreducible $\mathbb{Q}$-representations of $G$ whose kernels contain $G^{'}$ is in a bijective correspondence with the Galois conjugacy classes of complex degree $1$ representations. The number of such irreducible $\mathbb{Q}$-representations is equal to $4(n-1)$.
\item[2.] The set of all inequivalent irreducible $\mathbb{Q}$-representations of $G$ whose kernels do not contain $G^{'}$ is in a
bijective correspondence with the set of all monic irreducible
polynomials $\Phi_{1}(X), \Phi_{2}(X), \dots, \Phi_{2^{n-k}}(X)$,
where $k$ runs over $\{2, \dots, n\}$.
\item[3.] The total number of inequivalent irreducible $\mathbb{Q}$-representations of $G$ equals $4(n-1) + \frac{n(n+1)}{2}$.
\end{enumerate}
\end{thm}

\subsection{Rational Wedderburn decomposition of Hol$(C_{2^{n}})$}

We may use the previous results and the information on Schur indices
to find the Wedderburn decomposition of Hol$(C_{2^{n}})$ over
$\mathbb{Q}$. The result asserts:

\begin{thm}\label{Theorem $9.3.1$}
Let $G =$ Hol$(C_{2^{n}})$. For $m \geq 1$, let $\zeta_{m}$ denote a primitive $m$-th root of unity. The Wedderburn decomposition of $\mathbb{Q}[G]$ is
$$\mathbb{Q}[G] = \bigoplus_{d | 2^{n-2}}(\mathbb{Q}(\zeta_{d}))^{(2)} \bigoplus_{d | 2^{n-2}}(\mathbb{Q}(\zeta_{d}))^{(2)} \bigoplus^{n-k}_{l = 0} (\bigoplus^{n}_{k = 2} M_{2^{k-1}}(\mathbb{Q}(\zeta_{2^{l}}))).$$
\end{thm}

The proof is exactly the same as for the odd case Theorem \ref{Theorem $8.0.1$} on
using the counterparts for $p=2$ proved above.

We illustrate the descriptions of irreducible rational
representations and the Wedderburn decomposition through the
following example when $p=2$.

\begin{example}
Let $G =$ \textrm{Hol}$(C_{2^5})$. $G$ has a presentation of the form:
\begin{center}
$\langle{a, b, c\,|\, a^{2^5} = b^{2^3} = c^2 = 1, b^{-1}ab = a^5, c^{-1}ac = c^{-1}, c^{-1}bc = b}\rangle.$
\end{center}
We have seen in Example \ref{example 2} that corresponding to the
orbit $\mathcal{O}_0$, there are sixteen irreducible representations
of degree $1$.\\ Over $\mathbb{Q}$, $X^{2^3} - 1 =
\phi_{1}(X)\phi_{2}(X)\phi_{2^2}(X)\phi_{2^3}(X) = (X -1)(X
+1)(X^{2} +1)(X^{4}+1)$ is the factorization into monic irreducible
polynomials. As a consequence, there are $8$ irreducible
$\mathbb{Q}$-representations whose kernels contain $G^{'}$.

We have also seen in Example \ref{example 2} that corresponding to
the orbit $\mathcal{O}_1$, there are sixteen irreducible
representations
of degree $1$. \\
Over $\mathbb{Q}$, $X^{2^3} - 1 =
\phi_{1}(X)\phi_{2}(X)\phi_{2^2}(X)\phi_{2^3}(X) = (X -1)(X
+1)(X^{2} +1)(X^{4}+1)$ is the factorization into monic irreducible
polynomials. As a consequence of that there are $8$ irreducible
$\mathbb{Q}$-representations whose kernels contain $G^{'}$.

So the total number of irreducible $\mathbb{Q}$-representations whose kernels contain $G^{'}$ is equal to $16$.

From Example \ref{example 2}, corresponding to the orbit
$\mathcal{O}_2$, there are eight irreducible complex representations
of degree $2$. From those irreducible complex representations we
obtain $4$ irreducible  representations over $\mathbb{Q}$, and their
degrees are $2$, $2$, $4$ and $8$ respectively. They are in a
bijective correspondence with the set of all monic irreducible
factors of  $X^{2^{3}} - 1,$ over $\mathbb{Q}$.

From the same Example \ref{example 2}, corresponding to the orbit
$\mathcal{O}_3$, there are four irreducible complex representations
of degree $4$. From those representations we obtain $3$ irreducible
representations over $\mathbb{Q}$, and their degrees are $4$, $4$
and $8$ respectively. They are in a bijective correspondence with
the set of all monic irreducible factors of $X^{2^{2}} - 1,$ over
$\mathbb{Q}$.

From Example \ref{example 2}, corresponding to the orbit
$\mathcal{O}_4$, there are two irreducible complex representations
of degree $8$. From those representations we obtain $2$ irreducible
representations over $\mathbb{Q}$, and their degrees are $8$ and $8$
respectively. They are in a bijective correspondence with the set of
all monic irreducible factors of  $X^{2} - 1,$ over $\mathbb{Q}$.

By Example \ref{example 2}, corresponding to the orbit
$\mathcal{O}_5$, there is only one irreducible complex
representation of degree $16$, and is realizable over $\mathbb{Q}$.

Finally, the Wedderburn decomposition of the group algebra
$\mathbb{Q}[Hol(C_{2^5})]$ is
\begin{align*}
\mathbb{Q}[G] \cong & \displaystyle \oplus_{d | 2^3}(\mathbb{Q}(\zeta_{d}))^{(2)} \displaystyle \oplus_{d | 2^3}(\mathbb{Q}(\zeta_{d}))^{(2)}
\oplus M_{2}(\mathbb{Q}) \oplus M_{2}(\mathbb{Q}) \oplus M_{2}(\mathbb{Q}(i)) \oplus M_{2}(\mathbb{Q}(\zeta_{8}))\\
& \oplus M_{2^2}(\mathbb{Q}) \oplus M_{2^2}(\mathbb{Q}) \oplus M_{2^2}(\mathbb{Q}(i))
\oplus M_{2^3}(\mathbb{Q}) \oplus M_{2^3}(\mathbb{Q})
\oplus M_{2^4}(\mathbb{Q}).
\end{align*}
\end{example}

\vskip 5mm

\section{For Hol$(C_{p^{n}})$, $p$ a Prime, $c(G)$, $q(G)$ and $p(G)$}

\vskip 5mm

In this final section, we compute $c(G), q(G)$ and $p(G)$ for $G =$
Hol$(C_{p^{n}})$, $p$ any prime. We shall use the earlier notations.
The first useful observation below is valid for all primes but the
proofs are different for the odd and even cases. In some of the
results below, we allow the cases Hol$(C_2)$ and Hol$(C_4)$ as well.

\begin{lem}
$G$ has a unique minimal normal subgroup and is of order $p$.
\end{lem}
\begin{proof}
We consider the two cases $p$ odd and $p =2$ separately.

{\it Case $p$ odd.} \\Clearly $G$ has a unique Sylow-$p$ subgroup,
and we denote it by $P$. Note that $P = \langle{a, b^{p-1}}\rangle
\cong C_{p^{n}}\rtimes C_{p^{n-1}}$. Let $N = \langle a \rangle
\cong C_{p^n}$ and $H$ be a minimal normal subgroup of $G$. If $H$
contains an element $z$ of order coprime to $p$ then $z^{-1}az =
a^i$ where $i\not\equiv 1\mod p$ (since the automorphisms of $N
\cong C_{p^n}$ of the form $a \mapsto a^{1+kp}$ are of order some
power of $p$, and so $z$ can not act by such an automorphism).
Therefore $a^{-1}z^{-1}az = a^{i-1}$ where $i-1\not\equiv 0\pmod p$.
As $H$ is normal, then $a^{-1}z^{-1}az\in H$. It follows that
$a^{i-1}\in H$ where $(i-1,p)=1$. It follows that $\langle
a^{i-1}\rangle = \langle a \rangle \subseteq H$. Also this
containment is proper since $H$ contains the element $z$ of order
coprime to $p$, whereas $|\langle a \rangle|= p^n$. But every
subgroup of $N = \langle a \rangle= C_{p^n}$ is normal in $G$. This
contradicts to $H$ being a minimal normal subgroup of $G$. Therefore
$H$ contains elements of order some power of $p$. This follows that
any minimal normal subgroup of $G$ is contained in the unique
Sylow-$p$ subgroup $P$ of $G$. It can be seen easily that the center
of $P$ has order $p$. As $P$ is a $p$-group, then the center of $P$
is the only minimal normal subgroup of $P$. Hence the center of $P$
is the unique minimal normal subgroup of $G$

{\it Case $p=2$.}\\
 We have seen that ${\rm Aut}(C_{2^n})$ is an
abelian $2$-group of order $2^{n-1}$ (of course for $n=1$, ${\rm
Aut}(C_{2^n})$ is trivial). Now we assume that $n > 1$. Then there
are exactly two fixed points of ${\rm Aut}(C_{2^n})$ in $N = \langle
a \rangle \cong C_{2^n}$, and they are the identity element and the
unique element $a^{2^{n-1}}$ of order $2$. So the center of $G$ is
$\langle{a^{2^{n-1}}}\rangle \cong C_{2}$. As $G$ is a $2$-group and
the center of $G$ is of order $2$, it follows that the center of $G$
is the unique minimal normal subgroup of $G$. This completes the
proof.
\end{proof}
\begin{lem}\label{Lemma $10.0.2$}
 $G$ has a unique faithful irreducible representation over $\mathbb{C}$ and is of degree $\phi(p^{n})$.
\end{lem}
\begin{proof}
As the center of $G$ is cyclic, it is well known that $G$ has a
faithful irreducible representation over $\mathbb{C}$. Let $\rho$ be
a faithful irreducible representation of $G$ over $\mathbb{C}$. Let
$N$ be the cyclic normal subgroup $\langle a \rangle \cong
C_{p^{n}}$.

In the case, $p=2$ and $n=1$, $|G|=2$ and the assertion is true.

Now we assume that $p$ is a prime and $n>1$. As the restriction of $\rho$ to $N$ is faithful, then it decomposes into faithful irreducible representations. By Clifford's theorem (see \cite{MR2270898}, Theorem $6.2$), $\rho\downarrow_N\cong l(\rho_1\oplus \rho_2\oplus \cdots \oplus \rho_k)$, where $\rho_1,\rho_2,\cdots,\rho_k$ are $G$-conjugates and $l$ is a positive integer. 
Since the faithful irreducible representations of $N$ over
$\mathbb{C}$ correspond to the generators of $N$, and for any two
generators, there is an automorphism of $N$ taking one generator to
the other. So all the faithful irreducible representations of $N$
over $\mathbb{C}$ are $G$-conjugates. Thus $\rho\downarrow_N\cong
l(\rho_1\oplus \rho_2\oplus \cdots \oplus \rho_k)$, where
$\rho_1,\rho_2,\cdots,\rho_k$ are all the faithful irreducible
representations of $N$ over $\mathbb{C}$. By Frobenius-reciprocity (see
\cite{MR2270898}, Lemma $5.2$), $\rho$ occurs as an irreducible
component of $\rho_i\uparrow^G$, for all $i$. As the action of ${\rm
Aut}(C_{p^n})$ on $C_{p^n}$ is faithful, then the centralizer of $N$
in $G$ is $N$ itself. Consequently, the inertia subgroup of $\rho_i$
is $N$, and therefore by Clifford's theorem (see \cite{MR2270898},
Theorem $6.2$) $\rho_i\uparrow^G$ is an irreducible representation.
So $\rho \cong \rho_1\uparrow^G \cong \dots \cong \rho_k\uparrow^G$
and $\rho\downarrow_N\cong (\rho_1\oplus \rho_2\oplus \cdots \oplus
\rho_k)$, where $\rho_1,\rho_2,\cdots,\rho_k$ are all the faithful
irreducible representations of $N$ over $\mathbb{C}$. Hence $\rho$ is the unique
faithful irreducible representation of $G$ over $\mathbb{C}$. This
compeltes the proof.
\end{proof}

\begin{lem}
The unique faithful irreducible representation of $G$ over $\mathbb{C}$ is realizable over $\mathbb{Q}$.
\end{lem}
\begin{proof}
Let $\rho$ be the unique faithful irreducible representation of $G$
over $\mathbb{C}$ with the character $\chi$. As $\rho$ is the unique
faithful irreducible representation, then $\chi$ is a rational
valued character. Consequently, the Galois conjugacy class of $\chi$
is $\chi$ itself. By Lemma \ref{Lemma $7.0.1$} and Lemma \ref{Lemma
$9.2.1$}, the Schur index of $\rho$ with respect to $\mathbb{Q}$ is
equal to $1$. So $\rho$ is realizable over $\mathbb{Q}$. This
completes the proof.
\end{proof}

We note the immediate corollary which is striking.

\begin{cor}
$G$ has a unique faithful irreducible representation over $\mathbb{Q}$.
\end{cor}
\vskip 5mm

\subsection{Computing $c(G)$ and $q(G)$}
\vskip 5mm

\begin{lem}
Let $\eta$ be a faithful quasi-permutation representation of $G$ over $\mathbb{C}$. Then $\eta$ contains the unique faithful irreducible representation of $G$ over $\mathbb{C}$ as an irreducible constituent.
\end{lem}
\begin{proof}
Let $\eta\cong \eta_1\oplus \eta_2 \oplus \cdots \oplus \eta_r$ be the direct sum decomposition into irreducible $\mathbb{C}$-representations of $G$. If no $\eta_i$ is faithful, then its kernel will contain the unique minimal normal subgroup of $G$. It follows that the unique minimal normal subgroup of $G$ is contained in the kernel of $\eta$, and so $\eta$ is not faithful, a contradiction. Thus some $\eta_i$ is faithful, and by Lemma \ref{Lemma $10.0.2$}, $\eta_i$ is equivalent to the unique faithful irreducible representation of $G$ over $\mathbb{C}$. This completes the proof.
\end{proof}
\begin{cor}
Let $\eta$ be a faithful quasi-permutation representation of $G$ over $\mathbb{Q}$. Then $\eta$ contains the unique faithful irreducible representation of $G$ over $\mathbb{Q}$ as an irreducible constituent.
\end{cor}
\begin{lem}\label{Lemma $5.0.7$}
 The character of the unique faithful irreducible representation of $G$ over $\mathbb{C}$ takes values $\phi(p^{n})$, $-p^{n-1}$ and $0$, where $\phi$ is the Euler's phi function.
\end{lem}
\begin{proof}
Let $\rho$ be the faithful irreducible representation of $G$ over $\mathbb{C}$, and $\chi$ be its corresponding character.
As we have seen in the proof of Lemma \ref{Lemma $10.0.2$}, $\rho$ is induced by a faithful irreducible $\mathbb{C}$-representation of the normal subgroup $N = \langle{a}\rangle \cong C_{p^n}$. So $\chi$ takes the value $0$ outside $H$. The restriction of $\rho$ to $N$ is direct sum of all the faithful irreducible $\mathbb{C}$-representations of $N$. So $\chi(a^s)$ is the sum of $s$-th powers of primitive $p^n$-th roots of unity. If $(s,p)=1$ then it is clear that $\chi(a)= \chi(a^s)$. So the values of $\chi$ on $N = \langle a\rangle$ are $\chi(1),\chi(a),\chi(a^p),\dots,\chi(a^{p^{n-1}})$.

If $n=1$, then $\chi(1)=\phi(p)=p-1$, $\chi(a) = $ the sum of primitive $p$-th roots of unity $= -1$. Note that whenever $(s, p) = 1$, $\chi(a^s) = -1$.

Now we assume that $n >1$. It is easy to see that for $i > 2$, the sum of primitive $p^i$-th roots of unity is equal to $0$. Notice that $\chi(a^{p^{n-i}})$ is the sum of primitive $p^{i}$-th roots of unity (with multiplicity $p^{n-i}$). So for $i\ge 2$, $\chi(a^{p^{n-i}}) = 0$. Consider the case $i=0$ and $i=1$. If $i=0$, then $\chi(a^{p^n})=\chi(1)=\phi(p^n)$, where $\phi$ is the Euler's $\phi$ function. If $n=1$, then $\chi(a^{p^{n-1}})$ is the sum of primitive $p$-th roots of unity with multiplicity $p^{n-1}$, and $\chi(a^{p^{n-1}}) = -p^{n-1}$. This completes the proof.
\end{proof}
\begin{lem}\label{lem2}
$c(G) = q(G) = p^{n}$.
\end{lem}
\begin{proof}
By Lemma \ref{Lemma $10.0.2$}, Lemma \ref{Lemma $7.0.1$} and Lemma
\ref{Lemma $9.2.1$}, $G$ has a unique faithful irreducible
representation over $\mathbb{C}$, and is of Schur index $1$ over
$\mathbb{Q}$. By Lemma \ref{Lemma $5.0.7$}, the minimum character
value of that representation is $-p^{n-1}$. By Corollary
\ref{cor1}, we get $c(G) = q(G) = p^{n}$.
\end{proof}
\vskip 5mm

\subsection{Computing $p(G)$}
\vskip 5mm

\begin{lem}\label{lem3}
$p(G) \leq p^{n}$.
\end{lem}
\begin{proof}
 We show that the subgroup ${\rm Aut}(C_{p^n})$ is core-free in $G$, i.e., it does not contain any non-trivial normal subgroup of $G$. Suppose that $K\le {\rm Aut}(C_{p^n})$ is a non-trivial normal subgroup of $G$. Then $K$ is normalized by $N = \langle a \rangle\cong C_{p^n}$.
As $N$ is normalized by $K$ and $K\cap N =1$, then $K$ centralizes
$N$. So every automorphism in $K$ fixes every element of $\langle
x\rangle$; it follows that $K=1$. Consider the permutation
representation of $G$ on the cosets of ${\rm Aut}(C_{p^n})$; there
are $p^n$ such cosets. Notice that the kernel of this permutation
representation is the core of ${\rm Aut}(C_{p^n})$, and which is
$1$. So this representation is faithful. Hence the minimum degree of
a faithful permutation representation of $G$ is not more than $p^n$.
This completes the proof.
\end{proof}
Finally, we have the following Theorem.

\begin{thm}\label{Theorem $10.2.2$}
$c(G) = q(G) = p(G) = p^{n}$.
\end{thm}
\begin{proof}
As $c(G)\leq q(G)\leq p(G)$, then by Lemma \ref{lem2} and Lemma \ref{lem3}, we get $c(G) = q(G) = p(G) = p^{n}$.
\end{proof}

{\bf Acknowledgments}

The first author expresses gratitude to Rahul Dattatraya Kitture for useful discussions. He also wishes to express thanks to both the institutes, Indian Statistical Institute, Bangalore Centre, Bangalore-India and Harish-Chandra Research Institute (HRI), Prayagraj (Allahabad)-India for giving all the facilities to complete this work.

\end{document}